\documentclass[review]{elsarticle}

\usepackage{color}
\usepackage{float}
\usepackage{amsmath}
\usepackage{amsthm}
\usepackage{epstopdf} % needed if you have eps figures in a pdflatex manuscript
\usepackage{mathptmx}      % use Times fonts if available on your TeX system
\usepackage{lineno,hyperref}
\usepackage{latexsym}
\usepackage{amssymb}
\usepackage{amsfonts}
\usepackage{verbatim}
\usepackage[latin1]{inputenc}
\usepackage{color}
\usepackage{subfigure}
\usepackage{graphicx}
\usepackage{mathrsfs}
\usepackage{float}
\allowdisplaybreaks \allowdisplaybreaks[4]%change line

\newtheorem{theorem}{Theorem}[section]
\newtheorem{remark}{Remark}[section]

\newtheorem{proposition}{Proposition}[section]
\newtheorem{lemma}{Lemma}[section]

\journal{Journal of \LaTeX\ Templates}

\begin{document}

\begin{frontmatter}

\title{Optimal Strong Convergence Rate of a Backward Euler Type Scheme for the Cox--Ingersoll--Ross Model Driven by Fractional Brownian Motion}
%\tnotetext[mytitlenote]{Fully documented templates are available in the elsarticle package on \href{http://www.ctan.org/tex-archive/macros/latex/contrib/elsarticle}{CTAN}.}

%% Group authors per affiliation:
\author[mymainaddress]{Jialin Hong}
\ead{hjl@lsec.cc.ac.cn}
\author[mymainaddress]{Chuying Huang\corref{mycorrespondingauthor}}
\ead{huangchuying@lsec.cc.ac.cn}
\cortext[mycorrespondingauthor]{Corresponding author}
\author[mysecondaryaddress]{Minoo Kamrani}
\ead{m.kamrani@razi.ac.ir}
\author[mymainaddress]{Xu Wang}
%% or include affiliations in footnotes:
\ead{wangxu@lsec.cc.ac.cn}

\address[mymainaddress]{Academy of Mathematics and Systems Science, Chinese Academy of Sciences, Beijing 100190, China; School of Mathematical Sciences, University of Chinese Academy of Sciences, Beijing 100049, China}
\address[mysecondaryaddress]{Department of Mathematics, Faculty of Sciences, Razi University, Kermanshah, Iran}

\begin{abstract}
In this paper, we investigate the optimal strong convergence rate of numerical approximations for the Cox--Ingersoll--Ross model driven by fractional Brownian motion with Hurst parameter $H\in(1/2,1)$. To deal with the difficulties caused by the unbounded diffusion coefficient, we study an auxiliary equation based on Lamperti transformation. By means of Malliavin calculus, we prove that the backward Euler scheme applied to this auxiliary equation ensures the positivity of the numerical solution, and is of strong order one. Furthermore, a numerical approximation for the original model is obtained and converges with the same order.
\end{abstract}


\end{frontmatter}

\section{Introduction}
The Cox--Ingersoll--Ross (CIR) model 
\begin{align}\label{model}
dr(t)&=\kappa (\theta-r(t))dt+\sigma\sqrt{r(t)}dB(t), \quad t\in (0,T]
\end{align}
is proposed in \cite{CIR85} to study the interest rate dynamics and is also used in the Heston model to describe the stochastic volatility \cite{Heston93}. The constants $\theta$, $\sigma$ and $\kappa$ characterize the long-term mean, the volatility and the speed of adjustment, respectively. 
In the classical case, $B$ is assumed to be the standard Brownian motion and there is an amount of articles devoted to numerical approximations and their strong convergence rates for \eqref{model}; see  \cite{A05,A13,BBD08,DNS12,GR11,HMS13,JKN09NM,NS14} and references therein.
According to the memory phenomena in the real market, an appropriate modification (see e.g. \cite{Hes18,HNS08,Rev68,Mishura1804}) for the CIR model is to replace the standard Brownian motion by the fractional Brownian motion (fBm) with Hurst parameter $H$. The goal of this paper is to give the first result on strong convergence rate for the numerical approximation of (1) when $H\in(\frac12,1)$ in which case \eqref{model} is understood as a pathwise Riemann--Stieltjes integral equation.

%In this paper, we consider strong approximations for the Cox--Ingersoll--Ross (CIR) model 
%\begin{align}\label{model}
%dr(t)&=\kappa (\theta-r(t))dt+\sigma\sqrt{r(t)}dB(t), \quad t\in (0,T].
%\end{align}
%In mathematical finance, the model is proposed by \cite{CIR85} to study the term structure of interest rates and is also used in the Heston model to describe the stochastic volatility . The constants $\theta$, $\sigma$ and $\kappa$ characterize the mean, the volatility and the speed of adjustment, respectively. 
%Different from the classical CIR model which is driven by a standard Brownian motion (see \cite{A05,BBD08,DNS12,GR11,HMS13,JKN09NM,NS14,YW17NA} and references therein),  $\{B(t)\}_{t\in[0,T]}$ here is assumed to be a fractional Brownian motion (fBm) with Hurst parameter $H>1/2$. In this case, the fBm simulates the long-range memory phenomena in the real market (see e.g., \cite{Rev68}) and equation \eqref{model} is interpreted in the Riemann--Stieltjes sense pathwisely.

Comparing with the standard Brownian setting, two main difficulties in studying equations driven by fBms are the correlation between increments and the lack of martingale property, which make the fundamental convergence theorem (see \cite{MT04}) invalid in the numerical analysis. To deal with the above difficulties and to gain the strong convergence rates in multiplicative noise cases with $H\in(1/2,1)$, a priori estimates for numerical solutions in H\"older spaces are needed (see e.g., \cite{AIHP,HHW,CN2017,HLN16AAP}). It leads to an essential assumption that the coefficients are bounded, which is even not satisfied in general by globally Lipschitz ones.

To deal with unbounded or non-globally Lipschitz diffusion coefficients in \eqref{model}, we use the Lamperti transformation $X(t)=\sqrt{r(t)}$ and study the following auxiliary equation with a singular drift and an additive noise 
\begin{align}\label{model2}
dX(t)&=\frac12\kappa \left(\frac{\theta}{X(t)}-X(t)\right)dt+\frac12\sigma dB(t), \quad t\in (0,T].
\end{align}
Note that the backward Euler scheme applied to \eqref{model2} has an explicit expression which inherits the positivity of the exact solution, due to the special structure of the auxiliary equation. In this sense, it is more efficient to apply the backward Euler scheme to \eqref{model2} than to \eqref{model}. We refer to \cite{A05,A13,DNS12,NS14} for the study of the standard Brownian case. 
For the fractional Brownian case with $H\in(\frac12,1)$, we give the Malliavin derivative of the exact solution and the boundedness of the inverse moments of the exact solution, utilizing the techniques in  \cite{HNS08}. Based on these a priori estimates, we prove that the backward Euler scheme applied to \eqref{model2}, as well as the corresponding numerical approximation to \eqref{model}, converges with order one in the uniformly strong sense. Then the boundedness of the inverse moments of the numerical solution is also obtained.

Combining with the piecewise linear interpolation, we get that the continuous interpolation of the numerical solution also admits positive values on $[0,T]$.
We obtain that the strong convergence rate of this continuous interpolation in $L^{\infty}([0,T];L^p(\Omega;\mathbb{R}))$ is $O\Big(h^H\sqrt{\log \frac{T}{h}}\Big)$, where $h$ is the time step size. This rate matches the results for equations driven by additive fBms in the cases of the globally Lipschitz coefficient \cite{N06} and the one-sided dissipative Lipschitz coefficient with polynomial growth \cite{AMO09}. The strong convergence rate is optimal in the sense that it achieves the H\"older regularity of fBms. 
In addition, we derive the Malliavin derivative of the numerical solution based on the specific form of the backward Euler scheme, which serves as a counterpart with that of the exact solution.

The paper is organized as follows. In Section \ref{sec2}, we introuduce properties and Malliavin calculus with repsect to the fBm. We give the well-posedness and the boundedness of moments of the exact solution in Section \ref{sec3}. The strong convergence rates of the backward Euler type scheme and the continuous interpolation are obtained in Section \ref{sec4}, as well as the boundedness of the inverse moments and the Malliavin derivative of the numerical solution. We  present numerical experiments in Section \ref{sec5} to confirm our theoretical analysis.

\section{Preliminaries}\label{sec2}
In this section, we introduce some definitions and Malliavin calculus with respect to the fBm.
We refer to \cite{DU99PA,Nualart} for more details.
Throughout the paper, we use $C$ as a generic constant and use $C(\cdot)$ if necessary to mention the parameters it depends on.

The fBm $\{B(t)\}_{t\in[0,T]}$ on some probability space $(\Omega,\mathcal{F},\mathbb{P})$ is a centered Gaussian process with continuous sample paths and covariance
\begin{align*}
Cov(t,s):=\mathbb{E}\Big[B(t)B(s)\Big]=\frac12\bigg(t^{2H}+s^{2H}-|t-s|^{2H}\bigg),\quad \forall\ s,t\in [0,T],
\end{align*}
where $H\in(0,1)$ is the Hurst parameter. As a consequence, we have the $H$-H\"older regularity of the fBm in $L^p({\Omega;\mathbb{R}})$, i.e., 
\begin{align*}
\sup_{0\le s<t\le T}\frac{\|B(t)-B(s)\|_{L^p({\Omega;\mathbb{R}})}}{|t-s|^H}\le C(p), \quad \forall ~ p\ge 1.
\end{align*}
Further, we also have the $(H-\epsilon)$-H\"older regularity of continuous trajectories, i.e., for all $0<\epsilon<H$, there exists a nonnegative random variable $G(\epsilon,T)\in L^p(\Omega;\mathbb{R})$, $p\ge 1$, such that 
\begin{align*}
\sup_{0\le s<t\le T}\frac{|B(t)-B(s)|}{|t-s|^{H-\epsilon}}\le G(\epsilon,T), \quad {\rm a.s.}
\end{align*}

In the following, we always assume $H>1/2$. In this case, 
\begin{align*}
Cov(t,s)=\int_{0}^{t}\int_{0}^{s}\phi(\tau,u)dud\tau
\end{align*}
with  
$\phi(\tau,u):=\alpha_H|u-\tau|^{2H-2}$ and $\alpha_H:=H(2H-1)>0$. Define an inner product by 
\begin{align*}
<\mathbf 1_{[0,t]},\mathbf 1_{[0,s]}>_{\mathcal{H}}:=Cov(t,s).
\end{align*}
Denote by $\left(\mathcal{H},<\cdot,\cdot>_{\mathcal{H}}\right)$ the Hilbert space which is the closure of the space of all step functions on $[0,T]$ with respect to $<\cdot,\cdot>_{\mathcal{H}}$. The map $\mathbf 1_{[0,t]}\mapsto B(t)$ can be extended to an isometry from $\mathcal{H}$ onto the Gaussian space $\mathcal{H}_1$ associated with $B$ (\cite[Section 5]{Nualart}). More precisely, denote this isometry by $\varphi  \mapsto B(\varphi)$, then for any $\varphi,\psi\in\mathcal{H}$,
\begin{align*}
<\varphi,\psi>_{\mathcal{H}}=\mathbb{E}\Big[B(\varphi)B(\psi)\Big].
\end{align*}

Introduce the kernel $K_H(t,s):=c_Hs^{\frac12-H}\int_{s}^{t}(u-s)^{H-\frac32}u^{H-\frac12}du\mathbf 1_{\{s<t\}}$ with $c_H:=\sqrt{\frac{\alpha_H}{\beta(2-2H,H-\frac12)}}$ and $\beta(\cdot,\cdot)$ being the Beta function. Then the covariance has the expression
\begin{align*}
Cov(t,s)=\int_{0}^{t\wedge s}K_H(t,u)K_H(s,u)du.
\end{align*}
Define the operator $K^*$ from $\mathcal{H}$ to $L^2([0,T];\mathbb{R})$ by  $(K^*\varphi)(s):=\int_{s}^{T}\varphi(t) \frac{\partial K_H}{\partial t}(t,s)dt$.
Another isometry is then obtained between $\mathcal{H}$ and a closed subspace of $L^2([0,T];\mathbb{R})$ (\cite[Section 5]{Nualart}), i.e., 
\begin{align}\label{iso}
<\varphi,\psi>_{\mathcal{H}}=<K^*\varphi,K^*\psi>_{L^2([0,T];\mathbb{R})}.
\end{align}
Define the operator $K$ by $(K\upsilon)(t):=\int_{0}^{t}K_H(t,s)\upsilon(s)ds$, for $\upsilon\in L^2([0,T];\mathbb{R})$. Then we have that the image space $\mathcal{H}_H:=K(L^2([0,T];\mathbb{R}))$ is the associated Cameron--Martin space (\cite{DU99PA}). The injection from $\mathcal{H}$ to $\mathcal{H}_H$ is denoted by ${\rm R_H}:=K\circ K^*$, i.e., 
\begin{align*}
({\rm R_H}\varphi)(t):=\int_{0}^{t}K_H(t,s)(K^*\varphi)(s)ds.
\end{align*}

For any random variable 
\begin{align}\label{F}
F=f(B(t_1),\cdots,B(t_N)),
\end{align}
where $N\in \mathbb{N_+}$,  $t_1,\cdots,t_N\in[0,T]$ and $f\in\mathcal{C}^\infty_b(\mathbb{R}^N)$ is bounded smooth function with all derivatives bounded, the derivative operator $D$ on $F$ is defined by 
\begin{align*}
D_{\cdot}F:=\sum_{i=1}^{N}\frac{\partial f}{\partial x_i}(B(t_1),\cdots,B(t_N))\mathbf 1_{[0,t_i]}(\cdot),
\end{align*}
which is an $\mathcal{H}$-valued random variable.
The Sobolev space $\mathbb{D}^{1,p}$, $p\ge 1$, is the closure of the set containing all random variables in the form of \eqref{F} with respect to the norm 
\begin{align*}
\|F\|_{\mathbb{D}^{1,p}}:=\left( \mathbb{E}[|F|^p] + \mathbb{E}[\|DF\|_{\mathcal{H}}^p]  \right)^{\frac1p}.
\end{align*}

Consider the adjoint operator $\delta$ of the derivative operator $D$. If an $\mathcal{H}$-valued random variable $\varphi\in L^2(\Omega;\mathcal{H})$ satisfies 
\begin{align*}
\left| \mathbb{E}[<\varphi,DF>_{\mathcal{H}}]  \right|\le C(\varphi) \|F\|_{L^2(\Omega;\mathbb{R})},  \quad \forall~F\in \mathbb{D}^{1,2},
\end{align*}
then $\varphi\in {\rm Dom}(\delta)$ and $\delta(\varphi)\in L^2(\Omega;\mathbb{R})$ is characterized by
\begin{align*}
\mathbb{E}[<\varphi,DF>_{\mathcal{H}}]=\mathbb{E}[F\delta(\varphi)], \quad \forall~F\in \mathbb{D}^{1,2}.
\end{align*}
In fact, the definition of $\mathbb{D}^{1,p}$ can be extended to $\mathbb{D}^{1,p}(\mathcal{H})$ for $\mathcal{H}$-valued random variables and the space $\mathbb{D}^{1,2}(\mathcal{H})\subset{\rm Dom}(\delta)$ (\cite[Section 1.3]{Nualart}).

The Skorohod integral of $\varphi$ with respect to the fBm is defined by $\int_{0}^{T}\varphi(t) \delta B(t):=\delta(\varphi)$ and we have $\mathbb{E}\left[\int_{0}^{T}\varphi(t) \delta B(t)\right]=0$. 
The following lemma gives a precise estimate for the Skorohod integral in  $L^p(\Omega;\mathbb{R})$.

\begin{lemma}(\cite[Section 5.2.2]{Nualart})\label{Lp}
	Let $p>1$. Denote $\mathbb{D}^{1,p}(|\mathcal{H}|):=\{\varphi\in\mathbb{D}^{1,p}(\mathcal{H}):   \| \mathbb{E}[\varphi]\|^p_{|\mathcal{H}|}+ \mathbb{E}\left[  
	\|D\varphi\|^p_{|\mathcal{H}|\otimes|\mathcal{H}|}\right]<\infty\}$, where 
	\begin{align*}
	\| \varphi\|^2_{|\mathcal{H}|}:=\int_{[0,T]^2}|\varphi(\tau)||\varphi(u)|&\phi(\tau,u)d\tau du, \\
	\|D\varphi\|^2_{|\mathcal{H}|\otimes|\mathcal{H}|}:=\int_{[0,T]^4}|D_{\tau_1}\varphi(u_1)||D_{\tau_2}\varphi(u_2)|
	&\phi(\tau_1,\tau_2)\phi(u_1,u_2)d\tau_1du_1d\tau_2du_2. 
	\end{align*}
	If a process $\varphi\in\mathbb{D}^{1,p}(|\mathcal{H}|)$ belongs to the domain of $\delta$ in $L^p(\Omega;\mathbb{R})$, then 
	\begin{align*}
	\mathbb{E}\left[  \left| \int_{0}^{T}\varphi(t) \delta B(t) \right| ^p\right]\le
	C(H,p) \left(   \| \mathbb{E}[\varphi]\|^p_{|\mathcal{H}|}+ \mathbb{E}\left[  
	\|D\varphi\|^p_{|\mathcal{H}|\otimes|\mathcal{H}|}\right]\right).
	\end{align*}
\end{lemma}

For $t\in [0,T]$, it is defined that $\int_{0}^{t}\varphi(u) \delta B(u):=\delta(\varphi \mathbf 1_{[0,t]})$. The maximal inequality for the Skorohod integral is also established.

\begin{lemma}(\cite[Section 5.2.2]{Nualart})\label{max}
	Let $pH>1$. Denote 
	$\mathbb{L}^{1,p}_H:=\{\varphi\in\mathbb{D}^{1,2}(|\mathcal{H}|):   \|\varphi\|_{p,1}<\infty\}$, where 
	\begin{align*}
	\|\varphi\|^p_{p,1}:=\int_{0}^{T} \mathbb{E}[|\varphi(u)|^p] du +
	\mathbb{E}\left[  \int_{0}^{T}  \left(\int_{0}^{T} | D_{\tau} \varphi(u) |^{\frac1H}   d\tau\right) ^{pH}   du       \right].
	\end{align*}
	If a process $\varphi\in\mathbb{L}^{1,p}_H$, then 
	\begin{align*}
	\mathbb{E}\left[  \sup_{0\le t\le T}\left| \int_{0}^{t}\varphi(u) \delta B(u) \right| ^p\right]\le
	C(T,H,p) \|\varphi\|^p_{p,1}.
	\end{align*}
\end{lemma}

Furthermore, the relationship between the Skorohod integral and the Riemann--Stieltjes integral can be characterized by
\begin{align}\label{int}
\int_{0}^{t} \varphi(u) dB(u)=\int_{0}^{t} \varphi(u) \delta B(u)+\int_{0}^{t} \int_{0}^{T}D_{\tau}\varphi(u)\phi(\tau,u) d\tau du.
\end{align}

\section{The CIR Model Driven by FBm}\label{sec3}
Consider the CIR model
\begin{align}\label{eq1}
dr(t)&=\kappa (\theta-r(t))dt+\sigma\sqrt{r(t)}dB(t), \quad t\in (0,T];\\
r(0)&=r_0,\nonumber
\end{align}
where $\kappa\theta>0$, $\sigma>0$, $r_0>0$ is the deterministic initial value, and $\{B(t)\}_{t\in[0,T]}$ is an fBm with Hurst parameter $H>1/2$. The stochastic integration is treated in the Riemann--Stieltjes sense and the ordinary chain rule holds. Utilizing the Lamperti transformation $X(t)=\sqrt{r(t)}$, we have an auxiliary equation
\begin{align}\label{eq2}
dX(t)&=\frac12\kappa \left(\frac{\theta}{X(t)}-X(t)\right)dt+\frac12\sigma dB(t), \quad t\in (0,T];\\
X(0)&=X_0:=\sqrt{r_0}.\nonumber
\end{align}

The next proposition gives the well-posedness of equation \eqref{eq2}, as well as the boundedness of moments of the exact solution.

\begin{proposition}\label{proX}
	Let $T>0$ and $p\ge 1$. There exists a unique solution $X$ of equation \eqref{eq2} satisfying 
	\begin{align*}
	\left(\mathbb{E}\left[\sup_{0\le t\le T}|X(t)|^p\right]\right)^{\frac1p}\le C(T,H,p,X_0,\kappa,\theta,\sigma).
	\end{align*}
	Moreover, if the following condition holds 
	\begin{align}\label{app}
	\kappa\theta \exp\left(\frac{\kappa s}{2}\right) \ge (p+1)\int_{0}^{s} \frac{\sigma^2}{2} \exp\left(\frac{\kappa \tau}{2}\right) \phi(\tau,s)d\tau, \quad \forall ~ 0\le s\le T,
	\end{align}
	then 
	\begin{align*}
	\sup_{0\le t\le T}\left(\mathbb{E}\left[|X(t)^{-1}|^{p}\right]\right)^{\frac1p}\le C(T,X_0,\kappa).
	\end{align*}
\end{proposition}

\begin{proof}
	Using the transformation $Z(t)=\exp\left(\frac{\kappa t}{2}\right)X(t)$, we have
	\begin{align}\label{eq3}
	dZ(t)&=\frac12\kappa\theta \exp\left(\kappa t\right)  \frac{1}{Z(t)} dt+\frac12\sigma  \exp\left(\frac{\kappa t}{2}\right) dB(t), \quad t\in (0,T];\\
	Z(0)&=X_0.\nonumber
	\end{align}
	One can check that equation \eqref{eq3} satisfies the assumptions (i)-(iii) proposed in \cite[Section 2]{HNS08}. As a result, it admits a global unique solution on $[0,T]$, which is strictly positive and never hits zero almost surely.
	
	For moments of the exact solution, it follows from \cite[Theorem 2.3 and Theorem 3.1]{HNS08} that 
	\begin{align*}
	\left(\mathbb{E}\left[\sup_{0\le t\le T}|Z(t)|^p\right]\right)^{\frac1p}\le C(T,H,p,X_0,\kappa,\theta,\sigma).
	\end{align*}
	Denote $f(x):=\frac{\kappa\theta}{2x} -\frac{\kappa x}{2}$ and then $f'(x)=-\frac{\kappa\theta}{2x^2}-\frac{\kappa}{2}$.
	Combining the Malliavin derivative of $X$
	\begin{align}\label{DeX}
	D_sX(t)=\frac12\sigma\exp\left(\int_{s}^{t}f'(X(\tau))d\tau\right)\mathbf 1_{[0,t]}(s),
	\end{align}
	the chain rule of the Malliavin derivative leads to
	\begin{align*}
	D_sZ(t)&=\exp\left(\frac{\kappa t}{2}\right)D_sX(t) \\
	&= \frac12\sigma\exp\left(\frac{\kappa s}{2}\right)\exp\left(\int_{s}^{t}\frac{-\kappa\theta\exp\left(\kappa \tau\right)}{2Z^2(\tau)}d\tau\right)\mathbf 1_{[0,t]}(s),
	\end{align*}
	then 
	\begin{align*}
	0\le D_sZ(t) \le \frac12\sigma\exp\left(\frac{\kappa s}{2}\right).
	\end{align*}
	For any $p\ge 1$ and $\epsilon>0$, similar to \cite[Proposition 3.4]{HNS08}, the chain rule applied to $\frac{1}{(\epsilon+Z(t))^p}$ yields
	\begin{align*}
	&\frac{1}{(\epsilon+Z(t))^p}\\
	=&\frac{1}{(\epsilon+X_0)^p}-p\int_{0}^{t}\frac{\kappa\theta \exp\left(\kappa s\right)  }{2Z(s)(\epsilon+Z(s))^{p+1}}ds-p\int_{0}^{t}\frac{\sigma  \exp\left(\frac{\kappa s}{2} \right) }{2(\epsilon+Z(s))^{p+1}}dB(s)\\
	\le& \frac{1}{(\epsilon+X_0)^p}-p\int_{0}^{t}\frac{\kappa\theta \exp\left(\kappa s\right)  }{2Z(s)(\epsilon+Z(s))^{p+1}}ds-p\delta\left(\frac{\sigma  \exp\left(\frac{\kappa (\cdot)}{2} \right) }{2(\epsilon+Z(\cdot))^{p+1}}\mathbf 1_{[0,t]}(\cdot)\right)\\
	&-p\int_{0}^{t}\int_{0}^{s}\left[-(p+1)\frac{\sigma  \exp\left(\frac{\kappa s}{2} \right) }{2(\epsilon+Z(s))^{p+2}}  D_{\tau}Z(s) \phi(\tau,s) \right]  d\tau ds,
	\end{align*}
	where we use the formula \eqref{int}.
	Therefore, 
	\begin{align*}
	&\mathbb{E}\left[\frac{1}{(\epsilon+Z(t))^p}\right]\\
	\le& \frac{1}{(\epsilon+X_0)^p}-p\mathbb{E}\left[\int_{0}^{t}\frac{\kappa\theta \exp\left(\kappa s\right)  -(p+1)\int_{0}^{s}\sigma  \exp\left(\frac{\kappa s}{2} \right)D_{\tau}Z(s) \phi(\tau,s)d\tau  }{2(\epsilon+Z(s))^{p+2}}ds\right]\\
	\le& \frac{1}{(\epsilon+X_0)^p}-p\mathbb{E}\left[\int_{0}^{t}\frac{\kappa\theta \exp\left(\kappa s\right)  -(p+1)\int_{0}^{s}\sigma \exp\left(\frac{\kappa s}{2} \right)\frac{\sigma}{2} \exp\left(\frac{\kappa \tau}{2}\right) \phi(\tau,s)d\tau }{2(\epsilon+Z(s))^{p+2}}ds\right].
	\end{align*}
	If constants $H,T,p,\kappa,\theta$ and $\sigma$ satisfy condition \eqref{app}, 
	then by taking $\epsilon\rightarrow0$, we have
	\begin{align*}
	\sup_{0\le t\le T}\left(\mathbb{E}\left[|Z(t)^{-1}|^{p}\right]\right)^{\frac1p}\le X_0^{-1}.
	\end{align*}
	We conclude the proof according to the relationship between $Z$ and $X$.
\end{proof}

\begin{remark}
	If $\kappa>0$, we give the following condition 
	\begin{align*}
	T^{2H-1}\le \frac{2\kappa\theta}{\sigma^2 H(p+1)}
	\end{align*}
	as a sufficient condition for \eqref{app}.
	If $\kappa<0$, a sufficient condition for \eqref{app} is 
	\begin{align*}
	s^{2H-1}\le \frac{2\kappa\theta\exp\left(\frac{\kappa s}{2}\right)}{\sigma^2 H(p+1)}, \quad \forall ~ 0\le s\le T.
	\end{align*}
	Note that the continuous function $y(s):= \frac{2\kappa\theta\exp\left(\frac{\kappa s}{2}\right)}{\sigma^2 H(p+1)}-s^{2H-1}$ satisfies $y(0)>0$ since $\kappa\theta>0$. There always exists $T>0$ satisfying \eqref{app}.
\end{remark}

\begin{remark}
	If equation \eqref{eq2} is driven by the standard Brownian motion, the sufficient condition for the boundedness of inverse  moments of $X$ depends on $p,\kappa,\theta,\sigma$, but is independent of $T$. We refer to \cite[Section 3]{DNS12} for more discussions.
\end{remark}

%\begin{remark}
%	When $H<1/2$, equation \eqref{eq1} is in the sense of rough path and the behavior of the exact solution is not completely clear. According to the simulations in \cite{Mishura1804}, there seems to be non-zero probability for the exact solution hitting zero.
%\end{remark}

\section{A Backward Euler Type Scheme for the CIR Model}\label{sec4}
In this section, we apply the backward Euler scheme  to the auxiliary equation \eqref{eq2}, which is also suggested in \cite{A05,A13,DNS12,NS14} for the CIR model driven by standard Brownian motion. Then the numerical solution for the original equation \eqref{eq1} is obtained from the inverse of the Lamperti transformation. In this sense, we call it a backward Euler type scheme for the CIR model.

We first show that the scheme ensures the positivity of the numerical solution. 
Given $N\in\mathbb{N_+}$,  we denote by $h:=T/N$ the time step size. 
The backward Euler scheme applied to equation \eqref{eq2} is 
\begin{align}\label{sch1}
X_{n+1}=X_n+\frac12\kappa \left(\frac{\theta}{X_{n+1}}-X_{n+1}\right)h+\frac12 \sigma \Delta B_{n+1},
\end{align}
where $\Delta B_{n+1}=B(t_{n+1})-B(t_n)$, $t_n=nh$, $n=0,\cdots,N$. Since $\kappa\theta>0$, equation \eqref{sch1} has a unique positive solution for any $h>0$ satisfying $h\max\{0,-\kappa/2\}<1$, which is 
\begin{align}\label{sch2}
X_{n+1}=\frac{X_n+\frac12\sigma\Delta B_{n+1}+\sqrt{\left(  X_n+\frac12\sigma \Delta B_{n+1}\right)^2+\kappa h\theta \left(2+\kappa h\right)}}{2+\kappa h}.
\end{align}
Furthermore, we apply the piecewise linear interpolation to get the continuous interpolation $X^h$ on $[0,T]$. Namely, 
\begin{align}\label{sch3}
X^h(t)&=X_n+\frac{t-t_n}{h}(X_{n+1}-X_n), \quad \forall~t\in(t_n,t_{n+1}],~n=0,\cdots,N-1,\\
X^h(0)&=X_0.\nonumber
\end{align}
Consequently, we have $X^h$ positive on $[0,T]$.

\subsection{Optimal Strong Convergence Rate}
In the following, we give the strong convergence rate of the backward Euler scheme. Then the strong convergence rate of $r^h:=(X^h)^2$, which is an approximation of the solution to the original equation \eqref{eq1}, is obtained.

\begin{theorem} \label{thm1}
	Let $\xi\in(0,1)$. For $p\ge 2$, assume that 
	\begin{align}\label{appp}
	\kappa\theta \exp\left(\frac{\kappa s}{2}\right) \ge (3p+1)\int_{0}^{s} \frac{\sigma^2}{2} \exp\left(\frac{\kappa \tau}{2}\right) \phi(\tau,s)d\tau, \quad \forall ~ 0\le s\le T.
	\end{align}
	Then for any $h\in(0,1)$ satisfying $h{\max\{0,-\kappa/2\}}<1-\xi$, it holds that 
	\begin{align*}
	\left(\mathbb{E}\left[\sup_{1\le n \le N}|X(t_n)-X_n|^p\right]\right)^{\frac1p}\le C(T,H,p,X_0,\kappa,\theta,\sigma,\xi)h,
	\end{align*}
	where $X(\cdot)$ is the exact solution of equation \eqref{eq2} and $X_n$ is defined by scheme \eqref{sch2}.
\end{theorem}
\begin{proof}
	Denote $e_n:=X(t_n)-X_n$ and $f(x):=\frac12\kappa \left(\frac{\theta}{x}-x\right)$. 
	The chain rule and integration by parts of Riemann--Stieltjes integral yield
	\begin{align}\label{en+1}
	e_{n+1}=&e_{n}+  [f(X(t_{n+1}))-f(X_{n+1})] h  \\
	&-
	\int_{t_n}^{t_{n+1}}  \int_{t}^{t_{n+1}}   f'(X(u))[f(X(u))du+\frac12\sigma dB(u)]dt \nonumber\\
	=&e_{n}+  [f(X(t_{n+1}))-f(X_{n+1})] h  \nonumber\\
	&-
	\int_{t_n}^{t_{n+1}}  (u-t_n)   f'(X(u))[f(X(u))du+\frac12\sigma dB(u)] \nonumber\\
	=&:e_{n}+ [f(X(t_{n+1}))-f(X_{n+1})] h +R_{n}. \nonumber
	\end{align}
	Since $\kappa\theta>0$, $X(t_{n+1})>0$ and $X_{n+1}>0$, we have 
	\begin{align*}
	f(X(t_{n+1}))-f(X_{n+1})=\left( -\frac{\kappa\theta}{2X(t_{n+1})X_{n+1}}  -\frac{\kappa}{2} \right)e_{n+1}=:\gamma_{n+1}e_{n+1},
	\end{align*}
	where $\gamma_{n+1}\le \tilde{\kappa}$ with $0\le\tilde{\kappa}=\max\{0,-\frac{\kappa}{2}\}$.
	It follows that 
	\begin{align*}
	(1-\gamma_{n+1}h)e_{n+1}=e_n+R_{n}.
	\end{align*}
	
    Define $\zeta_0:=1$, $\zeta_n:=\prod_{j=1}^{n}(1-\gamma_j h)$, and $\tilde{e}_n:=\zeta_n e_n$.
    Multiplying the above formula by $\zeta_n$, we have 
    \begin{align*}
   \tilde{e}_{n+1}=\tilde{e}_n+\zeta_nR_{n}.
    \end{align*}
	According to \eqref{int}, we split $R_{n}$ into three terms
	\begin{align}\label{Rn}
	\begin{split}
	R_n
	=&-\int_{t_n}^{t_{n+1}}  (u-t_n)   f'(X(u))f(X(u))du- \int_{t_n}^{t_{n+1}} \frac12\sigma   (u-t_n)   f'(X(u))  \delta B(u)\\
	&-\int_{t_n}^{t_{n+1}}\int_{0}^{T}  \frac12\sigma   (u-t_n)  D_{\tau}[f'(X(u))]\phi(\tau,u) d\tau du.
	\end{split}
	\end{align}
	
	Defining $\tilde{\zeta}_n:=\zeta_n/(1-\tilde{\kappa} h)^n$, $\lfloor\frac{u}{h}\rfloor:=n$ for $u\in[t_n,t_{n+1})$, and 
	\begin{align*}
	\tilde{R}(t):= -\int_{0}^{t}  \frac12\sigma (1-\tilde{\kappa} h)^{\lfloor\frac{u}{h}\rfloor}  (u-t_{\lfloor\frac{u}{h}\rfloor})   f'(X(u))  \delta B(u),
	\end{align*}
	we obtain 
	\begin{align*}
	\tilde{\zeta}_n (\tilde{R}(t_{n+1})-\tilde{R}(t_n))=- \zeta_n \int_{t_n}^{t_{n+1}} \frac12\sigma   (u-t_n)   f'(X(u))  \delta B(u).
	\end{align*}
	Together with \eqref{Rn}, we deduce
	\begin{align*}
	\tilde{e}_n =& \tilde{e}_{n-1} - \int_{t_{n-1}}^{t_{n}} \zeta_{n-1} (u-t_{n-1})   f'(X(u))f(X(u))du
	+\tilde{\zeta}_{n-1} (\tilde{R}(t_{n})-\tilde{R}(t_{n-1}))\\
	&-\int_{t_{n-1}}^{t_{n}}\int_{0}^{T}  \frac12\sigma  \zeta_{n-1} (u-t_{n-1})  D_{\tau}[f'(X(u))]\phi(\tau,u) d\tau du\\
	=&-\int_{0}^{t_{n}} \zeta_{\lfloor\frac{u}{h}\rfloor} (u-t_{\lfloor\frac{u}{h}\rfloor})   f'(X(u))f(X(u))du+ \sum_{j=0}^{n-1} \tilde{\zeta}_j (\tilde{R}(t_{j+1})-\tilde{R}(t_j))\\
	&-\int_{0}^{t_n}\int_{0}^{T}  \frac12\sigma  \zeta_{\lfloor\frac{u}{h}\rfloor}   (u-t_{\lfloor\frac{u}{h}\rfloor})  D_{\tau}[f'(X(u))]\phi(\tau,u) d\tau du.
	\end{align*}
	Then the error $e_n$ satisfies
	\begin{align*}
	e_n = &-\int_{0}^{t_{n}} \frac{\zeta_{\lfloor\frac{u}{h}\rfloor}}{\zeta_n} (u-t_{\lfloor\frac{u}{h}\rfloor})   f'(X(u))f(X(u))du+ \sum_{j=0}^{n-1} \frac{\tilde{\zeta}_j}{\zeta_n} (\tilde{R}(t_{j+1})-\tilde{R}(t_j))\\
	&-\int_{0}^{t_n}\int_{0}^{T}  \frac12\sigma  \frac{\zeta_{\lfloor\frac{u}{h}\rfloor}}{\zeta_n} (u-t_{\lfloor\frac{u}{h}\rfloor})  D_{\tau}[f'(X(u))]\phi(\tau,u) d\tau du.
	\end{align*}
	Notice that 
	\begin{align*}
	\sum_{j=0}^{n-1} \tilde{\zeta}_j (\tilde{R}(t_{j+1})-\tilde{R}(t_j))=\tilde{\zeta}_{n-1}\tilde{R}(t_n)+\sum_{j=1}^{n-1} (\tilde{\zeta}_{j-1}-\tilde{\zeta}_{j}) \tilde{R}(t_j).
	\end{align*}
	Since $h$ satisfies $h\tilde{\kappa}<1-\xi$, we have $0<\xi<1-\tilde{\kappa} h\le1-\gamma_j h$ and then $0<\frac{\tilde{\zeta}_{j-1}}{\tilde{\zeta}_j}=\frac{1-\tilde{\kappa} h}{1-\gamma_j h}\le 1$. 
    Hence we obtain
	\begin{align*}
	\left|\sum_{j=0}^{n-1} \tilde{\zeta}_j (\tilde{R}(t_{j+1})-\tilde{R}(t_j))\right|
	&\le \tilde{\zeta}_{n-1}\left|\tilde{R}(t_n)\right|+\sum_{j=1}^{n-1} (\tilde{\zeta}_{j}-\tilde{\zeta}_{j-1}) \left|\tilde{R}(t_j)\right|\\
	&\le 2\tilde{\zeta}_{n} \sup_{1\le j\le n}\left|\tilde{R}(t_j)\right|.
	\end{align*}
	According to $\xi<1-\tilde{\kappa} h\le 1$, we get
	\begin{align*}
	\tilde{\zeta}_n/\zeta_n=(1-\tilde{\kappa}h)^{-n}=\left(1+\frac{\tilde{\kappa} h}{1- \tilde{\kappa} h}\right)^{n}\le e^{\frac{\tilde{\kappa}T}{\xi}}, \quad n=0,\cdots,N,
	\end{align*}
	which implies
	\begin{align*}
    \sup_{1\le n \le N} \left|\sum_{j=0}^{n-1} \frac{\tilde{\zeta}_j}{\zeta_n} (\tilde{R}(t_{j+1})-\tilde{R}(t_j))\right|
    \le 2 e^{\frac{\tilde{\kappa}T}{\xi}} \sup_{1\le j\le N}\left|\tilde{R}(t_j)\right|.
	\end{align*}
	
	Therefore, combining with $\frac{\zeta_{\lfloor\frac{u}{h}\rfloor}}{\zeta_n}=\prod_{j=\lfloor\frac{u}{h}\rfloor+1}^{n}(1-\gamma_jh)^{-1}\le (1-\tilde{\kappa} h)^{n-\lfloor\frac{u}{h}\rfloor}\le e^{\frac{\tilde{\kappa}T}{\xi}}$, we get
	\begin{align}
	&\mathbb{E}\left[\sup_{1\le n \le N}|e_n|^p\right] \nonumber\\
	\le  &C(p)\mathbb{E}\left[\sup_{1\le n \le N}\left|\int_{0}^{t_{n}}\frac{\zeta_{\lfloor\frac{u}{h}\rfloor}}{\zeta_n}   (u-t_{\lfloor\frac{u}{h}\rfloor})   f'(X(u))f(X(u))du\right|^p\right]\nonumber\\
	&+C(p)\mathbb{E}\left[\sup_{1\le n \le N}\left| \sum_{j=0}^{n-1} \frac{\tilde{\zeta}_j}{\zeta_n} (\tilde{R}(t_{j+1})-\tilde{R}(t_j))\right|^p\right]\nonumber\\
	&+C(p)\mathbb{E}\left[\sup_{1\le n \le N}\left|\int_{0}^{t_n}\int_{0}^{T}  \frac12\sigma  \frac{\zeta_{\lfloor\frac{u}{h}\rfloor}}{\zeta_n} (u-t_{\lfloor\frac{u}{h}\rfloor})  D_{\tau}[f'(X(u))]\phi(\tau,u) d\tau du\right|^p\right]\nonumber\\
	\le  &C(T,p,\kappa,\xi)\bigg(\mathbb{E}\left[\left(\int_{0}^{T} \left| (u-t_{\lfloor\frac{u}{h}\rfloor})   f'(X(u))f(X(u))\right|du\right)^p\right]\nonumber\\
	&+\mathbb{E}\left[\sup_{1\le j\le N}\left|\tilde{R}(t_j)\right|^p\right]+\mathbb{E}\left[\left(\int_{0}^{T}\int_{0}^{T} \left| \frac12\sigma  (u-t_{\lfloor\frac{u}{h}\rfloor})  D_{\tau}[f'(X(u))]\phi(\tau,u)\right| d\tau du\right)^p\right]\bigg)\nonumber\\
	=&:C(T,p,\kappa,\xi)(I+II+III),\label{en}
	\end{align}
	where the three terms $I,II,III$ are dealt with respectively in the following.
	
	Recall that $f(x)=\frac{\kappa\theta}{2x} -\frac{\kappa x}{2}$ and  $f'(x)=-\frac{\kappa\theta}{2x^2}-\frac{\kappa}{2}$.
	By Proposition \ref{proX}, the assumption \eqref{appp} implies
	\begin{align}\label{ap1}
	\mathbb E [X(t)^{-3p}]\le C(T,X_0,\kappa), \quad \forall~t\in[0,T].
	\end{align}
	According to the Minkowski inequality, the first term satisfies
	\begin{align*}
	I=&\left\|\int_{0}^{T} \left| (u-t_{\lfloor\frac{u}{h}\rfloor})   f'(X(u))f(X(u))\right|du\right\|^p_{L^{p}(\Omega;\mathbb{R})} \\
	\le&h^p \left( \int_{0}^{T} \|f'(X(u))f(X(u))\|_{L^{p}(\Omega;\mathbb{R})}  du\right)^p\\
	\le &C(T,p,X_0,\kappa)h^p.
	\end{align*}

	The chain rule of the Malliavin derivative gives that for any $u\in[0,T]$,
	\begin{align*}
	D_{\tau}[f'(X(u))]=f''(X(u))D_{\tau}X(u)=f''(X(u))\frac{\sigma}{2} \exp\left(\int_{\tau}^{u}f'(X(\iota))d\iota\right)\mathbf 1_{[0,u]}(\tau).
	\end{align*}
	Noticing that $f'(x)=-\frac{\kappa\theta}{2x^2}-\frac{\kappa}{2}\le -\frac{\kappa}{2}$ for $x>0$ and 
	$f''(x)=\frac{\kappa\theta}{x^3}$, we get 
	\begin{align*}
	0\le D_{\tau}[f'(X(u))]\le C(T,\kappa,\sigma) f''(X(u))\mathbf 1_{[0,u]}(\tau).
	\end{align*}
	Based on \eqref{ap1} and the fact 
	\begin{align*}
	\int_{0}^{T}\phi(\tau,u) d\tau \le 2HT^{2H-1},\quad \forall ~u\in[0,T],
	\end{align*}
	we obtain for the third term that
	\begin{align*}
	III=&\left\| \int_{0}^{T}\int_{0}^{T} \left| \frac12\sigma  (u-t_{\lfloor\frac{u}{h}\rfloor})  D_{\tau}[f'(X(u))]\phi(\tau,u)\right| d\tau du  \right\|^p_{L^{p}(\Omega;\mathbb{R})}\\
	\le& C(T,p,\kappa,\sigma)h^p\left( \int_{0}^{T} \left( \left\|f''(X(u))\right\|_{L^{p}(\Omega;\mathbb{R})}\int_{0}^{T}\phi(\tau,u) d\tau\right)du\right)^p  \\
	\le & C(T,H,p,X_0,\kappa,\theta,\sigma)h^p.
	\end{align*}
	
	Since $\frac12<H<1$, we know that $pH>1$ if $p\ge 2$. For the second term, 
	from Lemma \ref{max} and H\"older's inequality, we obtain
	\begin{align*}
	II=&\mathbb{E}\left[\sup_{1\le j\le N}\left|\tilde{R}(t_j)\right|^p\right]\\
	\le&C(T,H,p)  \int_{0}^{T} \mathbb{E}\left[\left|\frac12\sigma (1-\tilde{\kappa} h)^{\lfloor\frac{u}{h}\rfloor}  (u-t_{\lfloor\frac{u}{h}\rfloor})   f'(X(u)) \right|^p\right] du\\
	&+
	C(T,H,p)
	\mathbb{E}\left[  \int_{0}^{T}  \left(\int_{0}^{T}  \left| \frac12\sigma (1-\tilde{\kappa} h)^{\lfloor\frac{u}{h}\rfloor}  (u-t_{\lfloor\frac{u}{h}\rfloor}) D_{\tau}[ f'(X(u)) ]  \right|^{\frac1H}   d\tau\right) ^{pH}   du       \right]\\
	\le&C(T,H,p,\sigma) h^p \int_{0}^{T} \mathbb{E}\left[\left|f'(X(u)) \right|^p\right] du\\
	&+
	C(T,H,p,\sigma) h^p
	\mathbb{E}\left[  \int_{0}^{T}  \left(\int_{0}^{T}  \left|  D_{\tau}[ f'(X(u)) ]  \right|^{\frac1H}   d\tau\right) ^{pH}   du       \right]\\
	\le&C(T,H,p,X_0,\kappa,\theta,\sigma,\xi) h^p 
	+
	C(T,H,p,\sigma) h^p
	  \int_{0}^{T}  \mathbb{E}\left[  \left|  f''(X(u)  \right|^{p} \right]  du\\
	  \le& C(T,H,p,X_0,\kappa,\theta,\sigma,\xi)h^p.
	\end{align*}
	
	Finally, it turns out that 
	\begin{align*}
	\mathbb{E}\left[\sup_{1\le n \le N}|e_n|^p\right] \le C(T,H,p,X_0,\kappa,\theta,\sigma,\xi)h^p.
	\end{align*}
\end{proof}

%\begin{remark}
%	If $\kappa>0$, then there is no restriction on the step size $h$. If $\kappa<0$, it is required  $h<\frac{2(1-\xi)}{-\kappa}$, for some $0<\xi<1$, to ensure both the positivity and the strong convergence rate of the numerical solution.
%\end{remark}

\begin{remark}
	For the CIR model driven by standard Brownian motion, we refer to \cite{A13,NS14} and references therein for the converegence analysis. 
	For equations driven by globally Lipschitz drift terms and additive fBms, we refer to \cite{MC11,N06} for the convergence analysis.
\end{remark}

%\begin{remark}
%	For \eqref{en}, if we estimate $R_i$, $i=0,\cdots,N-1$, separately and then calculate the sum, it will lead to the strong convergence rate $H$, which is not optimal and cannot lead to the boundedness of inverse moments of the numerical solution in Proposition \ref{invX}. This implies that for equations driven by fBms, the optimal convergence rate cannot be obtained in general by estimating the local truncation error at first. Indeed, an efficient procedure is to analyze the global error directly. Under similar ideas for convergence analysis, we refer to \cite{N06} for the Euler scheme applied to equations with global Lipschitz drift and additive fBm, and to \cite{HHW,HLN16AAP} for Euler type schemes applied to equations with bounded drift and bounded multiplicative fBms.
%\end{remark}

Together with the H\"older regularity of the fBm, we obtain the following strong convergence rate for the continuous interpolation $X^h$ on $[0,T]$.

\begin{theorem}\label{thm2}
	Let $\xi\in(0,1)$. For $p\ge 2$, assume that 
	\eqref{appp} holds.
	Then for any $h\in(0,1)$ satisfying $h{\max\{0,-\kappa/2\}}<1-\xi$, it holds that 
	\begin{align*}
	\left(\mathbb{E}\left[\sup_{0\le t\le T}|X(t)-X^h(t)|^p\right]\right)^{\frac1p}\le C(T,H,p,X_0,\kappa,\theta,\sigma,\xi)h^H\sqrt{\log \frac{T}{h}},
	\end{align*}
	where $X(\cdot)$ is the exact solution of equation \eqref{eq2} and $X^h(\cdot)$ is defined by scheme \eqref{sch2}-\eqref{sch3}.
\end{theorem}
\begin{proof}
	By definition of $X^h$, for any $t\in(t_n,t_{n+1}]$,
	\begin{align*}
	X(t)-X^h(t)=&
	X(t)
	-\frac{t-t_n}{h}X_{n+1}-\frac{t_{n+1}-t}{h}X_n\\
	=&\frac{t_n-t}{h}\left[\int_{t}^{t_{n+1}}f(X(s))ds+\frac12\sigma(B(t_{n+1})-B(t))\right]\\
	&+\frac{t_{n+1}-t}{h}\left[\int_{t_n}^{t}f(X(s))ds+\frac12\sigma(B(t)-B(t_n))\right]\\
	&+\frac{t-t_n}{h}[X(t_{n+1})-X_{n+1}]+\frac{t_{n+1}-t}{h}[X(t_n)-X_n].
	\end{align*}
	Using the arguments in the proof of Theorem \ref{thm1} and the result in \cite[Theorem 6]{AAP03PLE} which shows 
	\begin{align*}
		&\left(\mathbb{E}\left[\sup_{0\le t\le T}\bigg|\frac{t_n-t}{h}(B(t_{n+1})-B(t) + \frac{t_{n+1}-t}{h}(B(t)-B(t_n)^p\right]\right)^{\frac1p}\\
		\le& C(T,H,p)h^H\sqrt{\log \frac{T}{h}},
	\end{align*}
	we conclude the result.
\end{proof}

\begin{remark}
	Based on $N$ observed values of the fBm, i.e., $\{B(t_1),\cdots,B(t_N)\}$, the strong convergence rate for the continuous interpolation of the numerical solution on $[0,T]$ in $L^p(\Omega;L^\infty([0,T];\mathbb{R}))$ is restricted by the H\"older regularity of the fBm.
	We refer to \cite{N06} for more discussions.
\end{remark}

Based on the inverse of the Lamperti transformation, we get the strong convergence rate for the backward Euler type scheme for the CIR model driven by fBm.

\begin{theorem}
	Let $\xi\in(0,1)$. For $p\ge 2$, assume that 
	\eqref{appp} holds.
	Then for any $h\in(0,1)$ satisfying $h{\max\{0,-\kappa/2\}}<1-\xi$, it holds that 
	\begin{align*}
\left(\mathbb{E}\left[\sup_{1\le n\le N}|r(t_n)-r^h(t_n)|^{\frac{p}{2}}\right]\right)^{\frac2p}\le C(T,H,p,X_0,\kappa,\theta,\sigma,\xi)h
\end{align*}
	and 
	\begin{align*}
	\left(\mathbb{E}\left[\sup_{0\le t\le T}|r(t)-r^h(t)|^{\frac{p}{2}}\right]\right)^{\frac2p}\le C(T,H,p,X_0,\kappa,\theta,\sigma,\xi)h^H\sqrt{\log \frac{T}{h}},
	\end{align*}
	where $r(\cdot)$ is the exact solution of equation \eqref{eq1} and $r^h(\cdot):=(X^h(\cdot))^2$.
\end{theorem}
\begin{proof}
     Based on Proposition \ref{proX} and Theorem \ref{thm2}, we have
     \begin{align*}
     &\left(\mathbb{E}\left[\sup_{1\le n\le N}|X(t_n)+X^h(t_n)|^{p}\right]\right)^{\frac1p}\\
     \le& \left(\mathbb{E}\left[\sup_{1\le n\le N}|2X(t_n)+(X^h(t_n)-X(t_n))|^{p}\right]\right)^{\frac1p}\\
     \le& C\left(\mathbb{E}\left[\sup_{1\le n\le N}|X(t_n)|^{p}\right]\right)^{\frac1p}+C
     \left(\mathbb{E}\left[\sup_{1\le n\le N}|X^h(t_n)-X(t_n)|^{p}\right]\right)^{\frac1p}\le C.
     \end{align*}
	Since 
	\begin{align*}
	r(t)-r^h(t)=(X(t)-X^h(t))(X(t)+X^h(t)),
	\end{align*}
	H\"older's inequality and Theorem \ref{thm1} yield
	\begin{align*}
	&\left(\mathbb{E}\left[\sup_{1\le n\le N}|r(t_n)-r^h(t_n)|^{\frac{p}{2}}\right]\right)^{\frac2p}\\
	\le& \left(\mathbb{E}\left[\sup_{1\le n\le N}|X(t_n)-X^h(t_n)|^{p}\right]\right)^{\frac1p}
	\left(\mathbb{E}\left[\sup_{1\le n\le N}|X(t_n)+X^h(t_n)|^{p}\right]\right)^{\frac1p}
	\\
	\le& C(T,H,p,X_0,\kappa,\theta,\sigma,\xi)h.
	\end{align*}
	Similarly, Theorem \ref{thm2} shows
	\begin{align*}
\left(\mathbb{E}\left[\sup_{0\le t\le T}|r(t)-r^h(t)|^{\frac{p}{2}}\right]\right)^{\frac2p}\le C(T,H,p,X_0,\kappa,\theta,\sigma,\xi)h^H\sqrt{\log \frac{T}{h}}.
\end{align*}
\end{proof}

\subsection{Inverse Moments and the Malliavin derivative of the Numerical Solution}
In this section, we give the boundedness of inverse moments of the numerical solution and derive the Malliavin derivative of the numerical solution, which are regarded as counterparts with properties of the exact solution.

\begin{proposition}\label{invX}
	Suppose that the assumptions in Theorem \ref{thm1} hold. Then 
	\begin{align*}
	\sup_{n=1,\cdots,N}\|X_{n}^{-1}\|_{L^p(\Omega;\mathbb{R})}\le C(T,H,p,X_0,\kappa,\theta,\sigma,\xi).
	\end{align*}
\end{proposition}
\begin{proof}
	From \eqref{sch1} and \eqref{en+1}, we have
	\begin{align*}
	X_{n+1}^{-1}=2(\kappa\theta h)^{-1} \left(e_n-(\frac12 \kappa h+1)e_{n+1}+\frac12\kappa\theta h (X(t_{n+1}))^{-1}+R_n\right).
	\end{align*}
	According to Proposition \ref{proX} and Theorem \ref{thm1}, it suffices to estimate the term involving $R_n$.
	 Together with the decomposition \eqref{Rn} of $R_n$, the techniques to deal with $I,II,III$ in \eqref{en} yield
	\begin{align*}
h^{-1}\left(\mathbb{E}\left[\sup_{1\le n\le N}|R_n|^{p}\right]\right)\le C,
\end{align*}
which gives the boundedness of inverse moments of the numerical solution.
\end{proof}

\begin{proposition}
	Suppose that the assumptions in Theorem \ref{thm1} hold. If $\kappa>0$, then for any $t\in(t_n,t_{n+1}]$, $n=0,\cdots,N-1$,
	\begin{align}\label{DeXh}
	D_s[X^h(t)]=\frac{t_{n+1}-t}{h}G_n(s)\mathbf 1_{[0,t_n]}(s)+\frac{t-t_n}{h}G_{n+1}(s)\mathbf 1_{[0,t_{n+1}]}(s)
	\end{align}
	with $G_n(s)=\frac12\sigma \prod_{j=i}^{n}(1-f'(X_{j})h)^{-1}$ for $s\in(t_{i-1},t_{i}]$,  $i=1,\cdots,n$.
\end{proposition}
\begin{proof}
	For any $\varphi\in \mathcal H$, noticing that the operator ${\rm R_H}$ is an injection from $\mathcal{H}$ to $\mathcal{H}_H$ (see e.g.  \cite{DU99PA}), we consider the following equation 
	\begin{align}\label{eq4}
	dX^\epsilon(t)&=\frac12\kappa \left(\frac{\theta}{X^\epsilon(t)}-X^\epsilon(t)\right)dt+\frac12\sigma dB(t)+\frac12\sigma \epsilon d\tilde{\varphi}(t), \quad t\in (0,T];\\
	X^\epsilon(0)&=X_0 \nonumber
	\end{align}
	with $\tilde{\varphi}={\rm R_H}\varphi$ and $\epsilon>0$. Based on the backward Euler scheme \eqref{sch2}, we get numerical solution $\{X^\epsilon_n\}_{n=0}^{N}$ for equation \eqref{eq4}.
	Then the directional derivarive of $X_n$ is given by
	\begin{align*}
	<DX_n,\varphi>_{\mathcal H}=\frac{{\rm d} X^\epsilon_n}{{\rm d} \epsilon} | _{\epsilon=0},\quad \varphi\in \mathcal H.
	\end{align*}
	The mean value theorem shows that for any $n=0,\cdots,N-1$,
	\begin{align*}
	X^\epsilon_{n+1}-X_{n+1}&=X^\epsilon_{n}-X_{n}+(f(X^\epsilon_{n+1})-f(X_{n+1}))h+\frac12\sigma \epsilon (\tilde{\varphi}(t_{n+1})-\tilde{\varphi}(t_{n}))\\
	&=X^\epsilon_{n}-X_{n}+f'(\tilde{X}^\epsilon_{n+1})\left(X^\epsilon_{n+1}-X_{n+1}\right)h+\frac12\sigma \epsilon (\tilde{\varphi}(t_{n+1})-\tilde{\varphi}(t_{n}))
	\end{align*}
	with $\tilde{X}^\epsilon_{n+1}=\tilde{h}^\epsilon_{n+1} X^\epsilon_{n+1}+(1-\tilde{h}^\epsilon_{n+1})X_{n+1}$ for some $\tilde{h}^\epsilon_{n+1}\in[0,1]$.
	Denoting $a^\epsilon_{n+1}:=(1-f'(\tilde{X}^\epsilon_{n+1})h)^{-1}$, we have 
	\begin{align*}
	X^\epsilon_{n+1}-X_{n+1}=a^\epsilon_{n+1}(X^\epsilon_{n}-X_{n})+\frac12\sigma \epsilon a^\epsilon_{n+1} (\tilde{\varphi}(t_{n+1})-\tilde{\varphi}(t_{n})).
	\end{align*} 
	Then 
	\begin{align*}
	X^\epsilon_{n}-X_{n}=\frac12\sigma \epsilon \sum_{i=1}^{n}\left(\prod_{j=i}^{n}a^\epsilon_{j} \right) (\tilde{\varphi}(t_{i})-\tilde{\varphi}(t_{i-1})).
	\end{align*} 
	Define $G_n^\epsilon(t):=\frac12\sigma \prod_{j=i}^{n}a^\epsilon_{j}$, for $t\in(t_{i-1},t_{i}]$ and $i=1,\cdots,n$. Recalling the definitions of ${\rm R_H}$ and $K^*$ introduced in Section \ref{sec2}, we obtain that 
	\begin{align*}
	X^\epsilon_{n}-X_{n}&=\epsilon \int_{0}^{t_n}G_n^\epsilon(\tau) d \tilde{\varphi}(\tau)\\
	&=\epsilon \int_{0}^{t_n}G_n^\epsilon(\tau) d {\rm R_H}\varphi(\tau)\\
	&= \epsilon \int_{0}^{t_n}G_n^\epsilon(\tau) \left[  \int_{0}^{\tau} \frac{\partial K_H (\tau,u)}{\partial \tau} (K^* \varphi)(u) du \right]  d\tau\\
	&= \epsilon \int_{0}^{t_n} \left[  \int_{u}^{t_n} G_n^\epsilon(\tau) \frac{\partial K_H (\tau,u)}{\partial \tau}  d\tau \right]  (K^* \varphi)(u)du\\
	&= \epsilon \int_{0}^{t_n} (K^* (G_n^\epsilon\mathbf 1_{[0,t_n]}))(u)(K^* \varphi)(u)du.
	\end{align*} 
	The isometry \eqref{iso} gives 
	\begin{align*}
	X^\epsilon_{n}-X_{n}= \epsilon <G_n^\epsilon\mathbf 1_{[0,t_n]},\varphi>_{\mathcal H}.
	\end{align*} 
	Note that $G_n^\epsilon$ is bounded due to the assumption $\kappa>0$ and the fact $f'(x)=-\frac12\kappa\theta\frac{1}{x^2}-\frac12\kappa<0$.  By the definition of scheme \eqref{sch2} and the dominated convergence theorem, we get that the following equality holds in both almost surely sense and $L^p(\Omega;\mathbb{R})$ sense for $p\ge 1$:
	\begin{align*}
	\lim_{\epsilon\rightarrow 0}\frac{X^\epsilon_{n}-X_{n}}{\epsilon}=\lim_{\epsilon\rightarrow 0} <G_n^\epsilon\mathbf 1_{[0,t_n]},\varphi>_{\mathcal H}=<G_n\mathbf 1_{[0,t_n]},\varphi>_{\mathcal H},
	\end{align*} 
	where $G_n(t):=\frac12\sigma \prod_{j=i}^{n}(1-f'(X_{j})h)^{-1}$ for any $t\in(t_{i-1},t_{i}]$ and $i=1,\cdots,n$. It follows from \cite[Theorem 3.1]{S85} that $X_n\in\mathbb{D}^{1,2}$ and $DX_n=G_n\mathbf 1_{[0,t_n]}$.
	
	By the definition of $X^h$, for any $t\in(t_n,t_{n+1}]$, we obtain the expression of the Malliavin derivative of $X^h$
	\begin{align*}
	D_s[X^h(t)]&=D_s\left[\frac{t_{n+1}-t}{h}X_n+\frac{t-t_n}{h}X_{n+1}\right]\\
	&=\frac{t_{n+1}-t}{h}G_n(s)\mathbf 1_{[0,t_n]}(s)+\frac{t-t_n}{h}G_{n+1}(s)\mathbf 1_{[0,t_{n+1}]}(s).
	\end{align*}
\end{proof}	

\begin{remark}\label{rk4}
	Notice $e^x\approx 1+x$ when $x$ is close to $0$. One can regard the expression of $D_s[X^h(t)]$ as an approximation for $D_s[X(t)]$ based on the backward Euler scheme, comparing \eqref{DeX} and \eqref{DeXh}.
\end{remark}

%\begin{remark}\label{rk3}
%Equation \eqref{eq3} is a singular equation since there is a singularity in the drift coefficient. From this point of view, the stability of the exact solution includes the boundedness of inverse moments. The expression  \eqref{DeX} of the Malliavin derivative of the exact solution plays a crucial role in the proof of the boundedness. For the numerical solution, we conjecture from numerical experiments that its inverse moments are bounded and independent of $T$. We leave the stability of the numerical solution as an open problem for the CIR model driven by fBm.
%\end{remark}

\section{Numerical Experiments}\label{sec5}
In the CIR model \eqref{eq1}, we set the initial value $r_0=1$, the constants $\kappa=2$, $\theta=0.5$ and $\sigma=0.5$. Figure \ref{f1} presents the pathwise behaviors of the numerical solutions for three cases $H=0.6,0.7,0.8$ with $T=10$ and $h=\frac{10}{2^{12}}=0.0024$. We observe that the solution is pushed forward along the direction from the initial value to the long-term mean value $\theta$ and the regularity of the solution is impacted by that of the fBm.

\begin{figure}
	\centering
	\subfigure[$H=0.6$]{
		\begin{minipage}[t]{0.3\linewidth}
			\includegraphics[height=3.6cm,width=3.6cm]{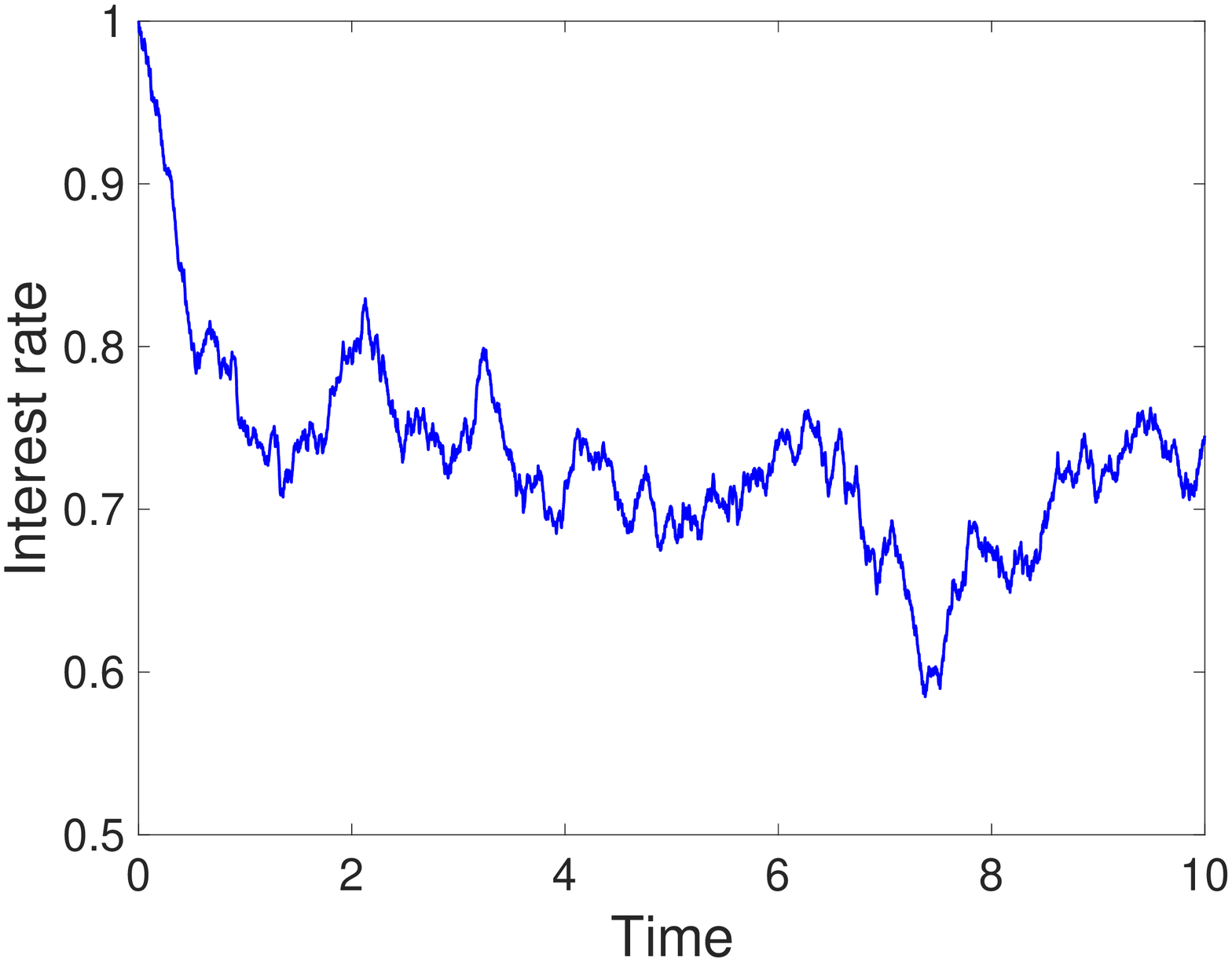}
		\end{minipage}
	}
	\subfigure[$H=0.7$]{
		\begin{minipage}[t]{0.3\linewidth}
			\includegraphics[height=3.6cm,width=3.6cm]{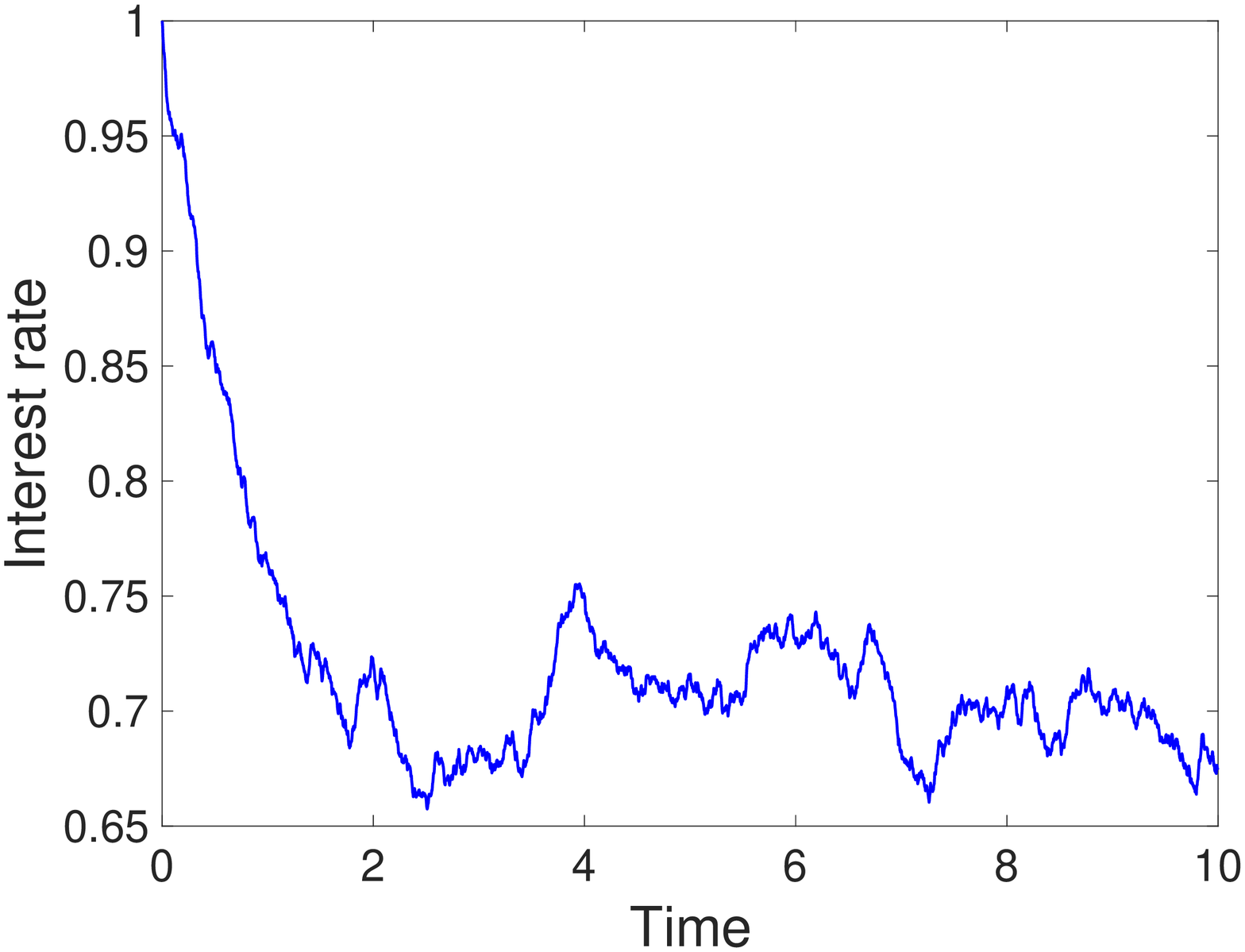}
		\end{minipage}
	}
	\subfigure[$H=0.8$]{
		\begin{minipage}[t]{0.3\linewidth}
			\includegraphics[height=3.6cm,width=3.6cm]{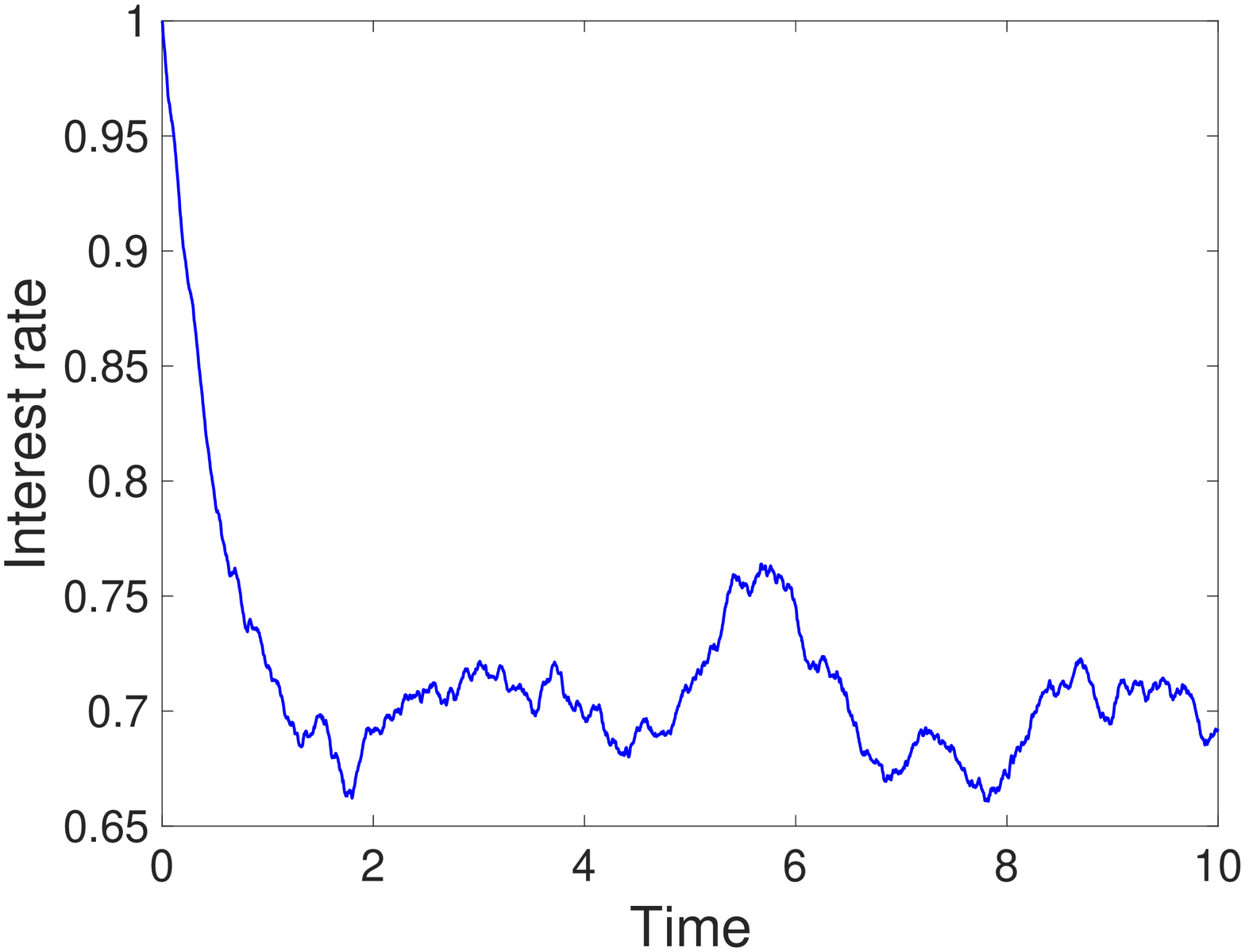}
		\end{minipage}
	}
	\caption{Interest rate $r^h$ vs. time}\label{f1}
\end{figure}

Since condition \eqref{app} is satisfied for $p=6$ and $T=1$,  we have from Theorem \ref{thm1} that 
\begin{align*}
\left(\mathbb{E}\left[\sup_{1\le n\le N}|X(t_n)-X_n|^{2}\right]\right)^{\frac12}\le C h.
\end{align*}
To verify the strong convergence order of the backward Euler scheme, we simulate the `exact' solution by using the small step size $h^*=2^{-15}$. The five step sizes $h=2^{-i}$, $i=6,7,8,9,10$ are chosen for numerical solutions. In Figure \ref{f2}, we calculate the mean square error with respect to $X_n$ by
\begin{align*}
\left(\mathbb{E}\left[\sup_{1\le n\le T/h}|X(t_n)-X_n|^{2}\right]\right)^{\frac12}
\end{align*}
and approximate the expectation by computing averages over 500 samples. The numerical results confirm our theoretical analysis.

\begin{figure}
	\centering
	\subfigure[$H=0.6$]{
		\begin{minipage}[t]{0.3\linewidth}
			\includegraphics[height=3.6cm,width=3.6cm]{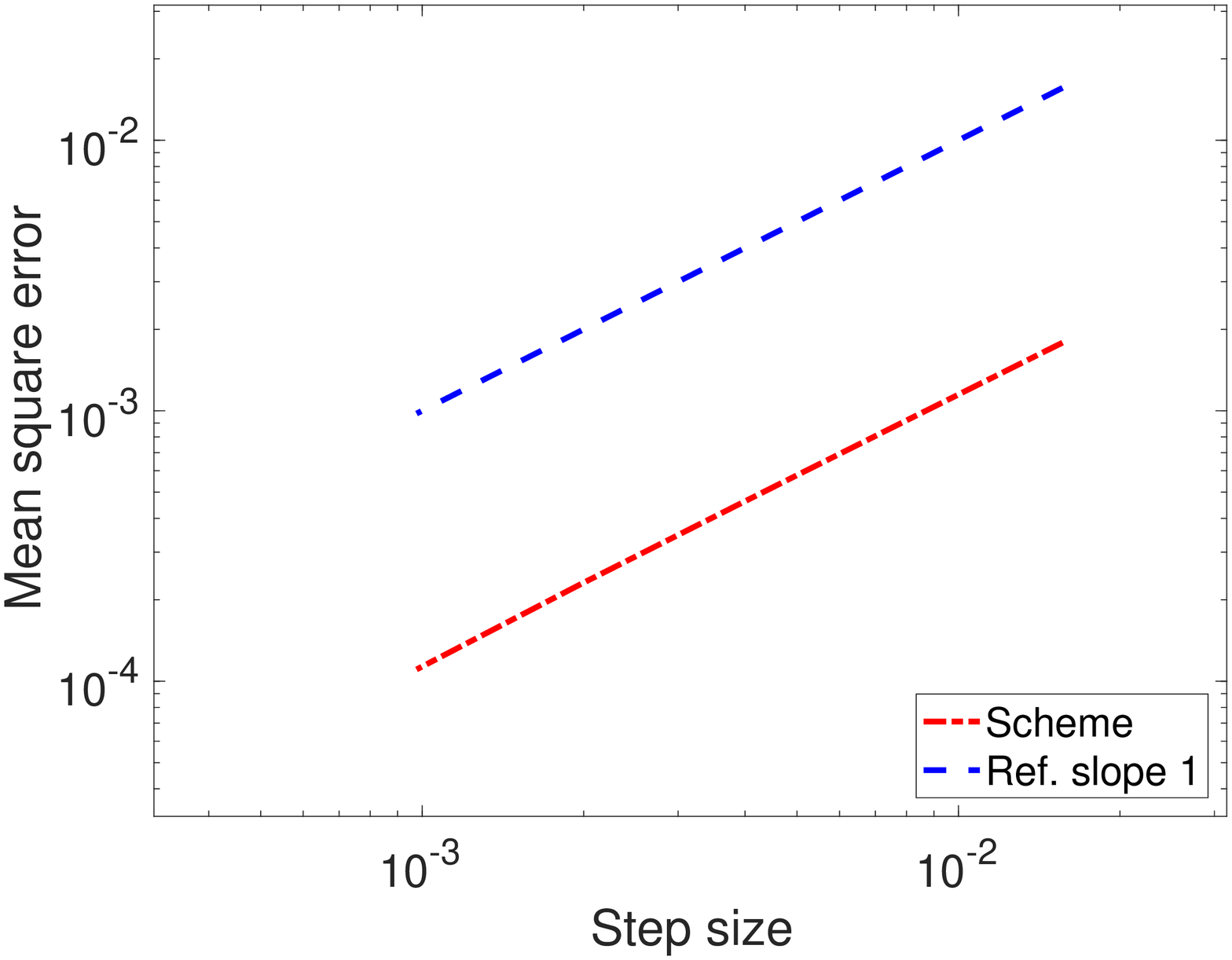}
		\end{minipage}
	}
	\subfigure[$H=0.7$]{
		\begin{minipage}[t]{0.3\linewidth}
			\includegraphics[height=3.6cm,width=3.6cm]{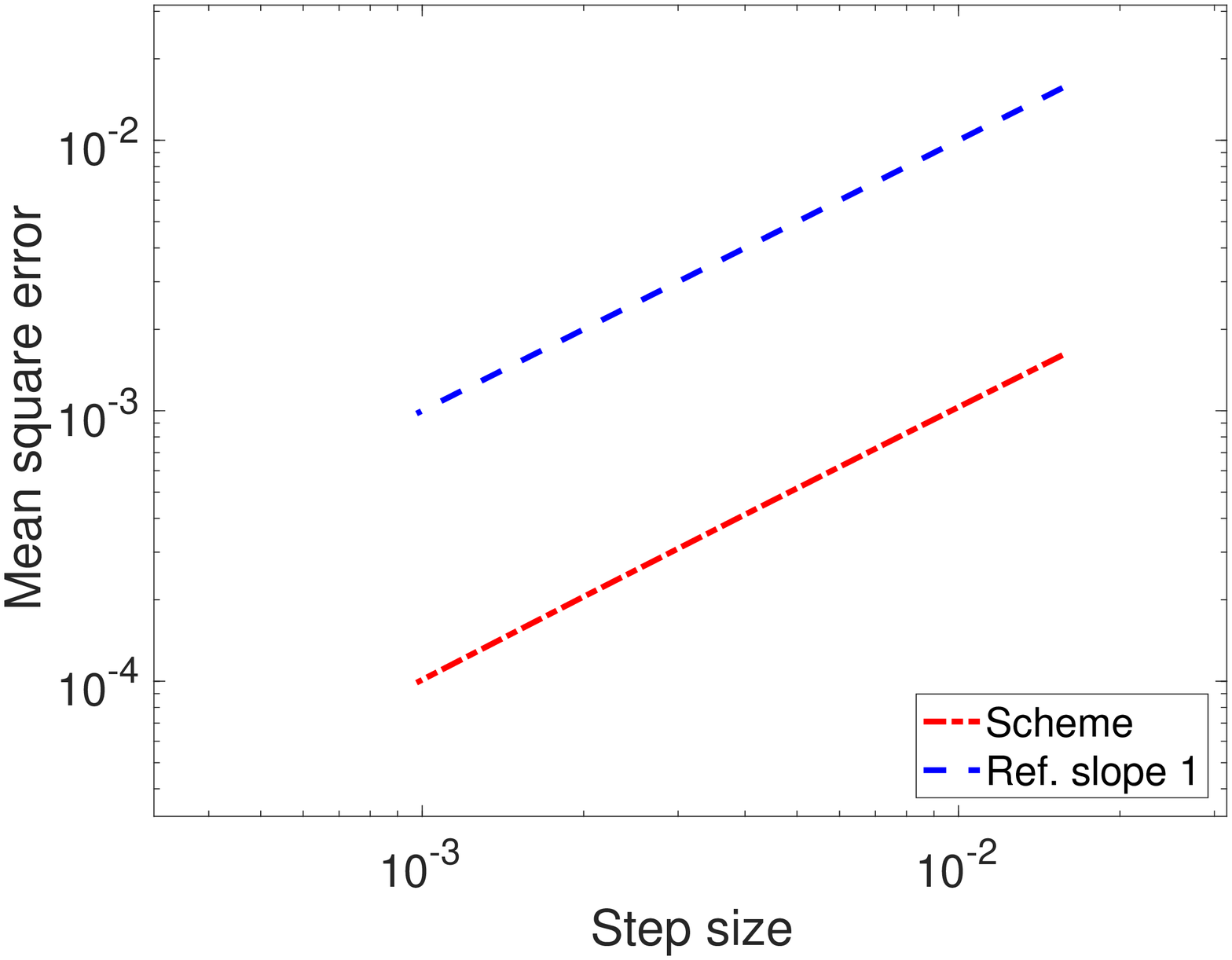}
		\end{minipage}
	}
	\subfigure[$H=0.8$]{
		\begin{minipage}[t]{0.3\linewidth}
			\includegraphics[height=3.6cm,width=3.6cm]{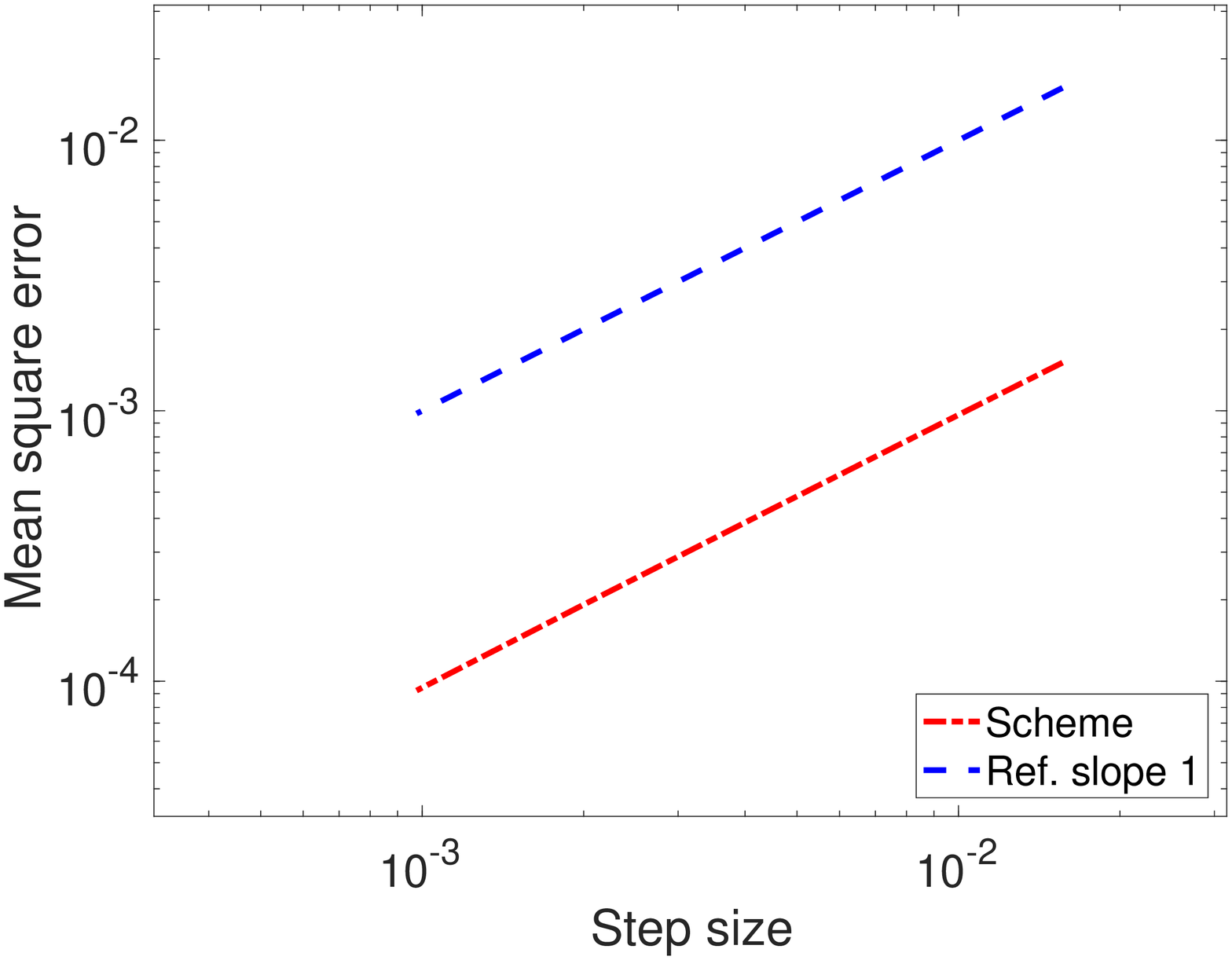}
		\end{minipage}
	}
	\caption{Mean square error for $X_n$ vs. step size }\label{f2}
\end{figure}

Under the same assumption, Theorem \ref{thm2} leads to 
\begin{align*}
\left(\mathbb{E}\left[\sup_{0\le t\le T}|X(t)-X^h(t)|^{2}\right]\right)^{\frac12}\le C h^{H}\sqrt{\log \frac{T}{h}}.
\end{align*}
To demonstrate the convergence rate on $[0,T]$, we also simulate the `exact' solution by using the small step size $h^*=2^{-15}$. The five step sizes $h=2^{-i}$, $i=6,7,8,9,10$ are chosen for numerical solutions. Then we apply the piecewise linear interpolation to get the numerical solutions $X^h(nh^*)$, $n=1,\cdots,2^{15}$. In Figure \ref{f3}, we calculate the mean square error with respect to $X^h$ by
\begin{align*}
\left(\mathbb{E}\left[\sup_{n=1,\cdots,2^{15}}|X(nh^*)-X^h(nh^*)|^{2}\right]\right)^{\frac12}
\end{align*}
and approximate the expectation by computing averages over 500 samples. In consistent with our theoretical analyisis, the strong convergence order for the piecewise linear interpolation is $H$ which is mainly caused by the H\"older regularity of fBm. Additionally, the error becomes smaller when $H$ becomes larger, due to the fact that the regularity of the fBm becomes better.

\begin{figure}
	\centering
	\subfigure[$H=0.6$]{
		\begin{minipage}[t]{0.3\linewidth}
			\includegraphics[height=3.6cm,width=3.6cm]{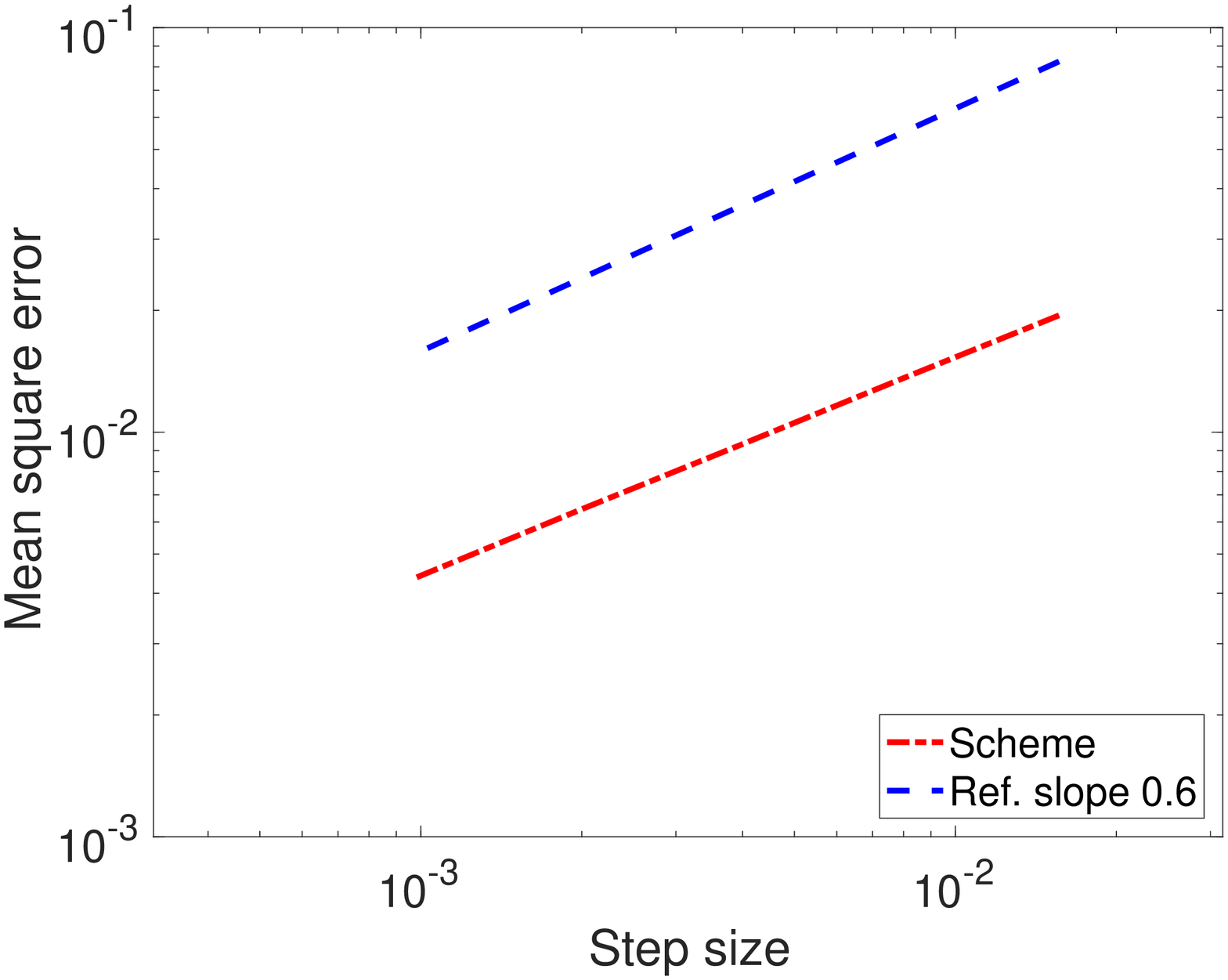}
		\end{minipage}
	}
	\subfigure[$H=0.7$]{
		\begin{minipage}[t]{0.3\linewidth}
			\includegraphics[height=3.6cm,width=3.6cm]{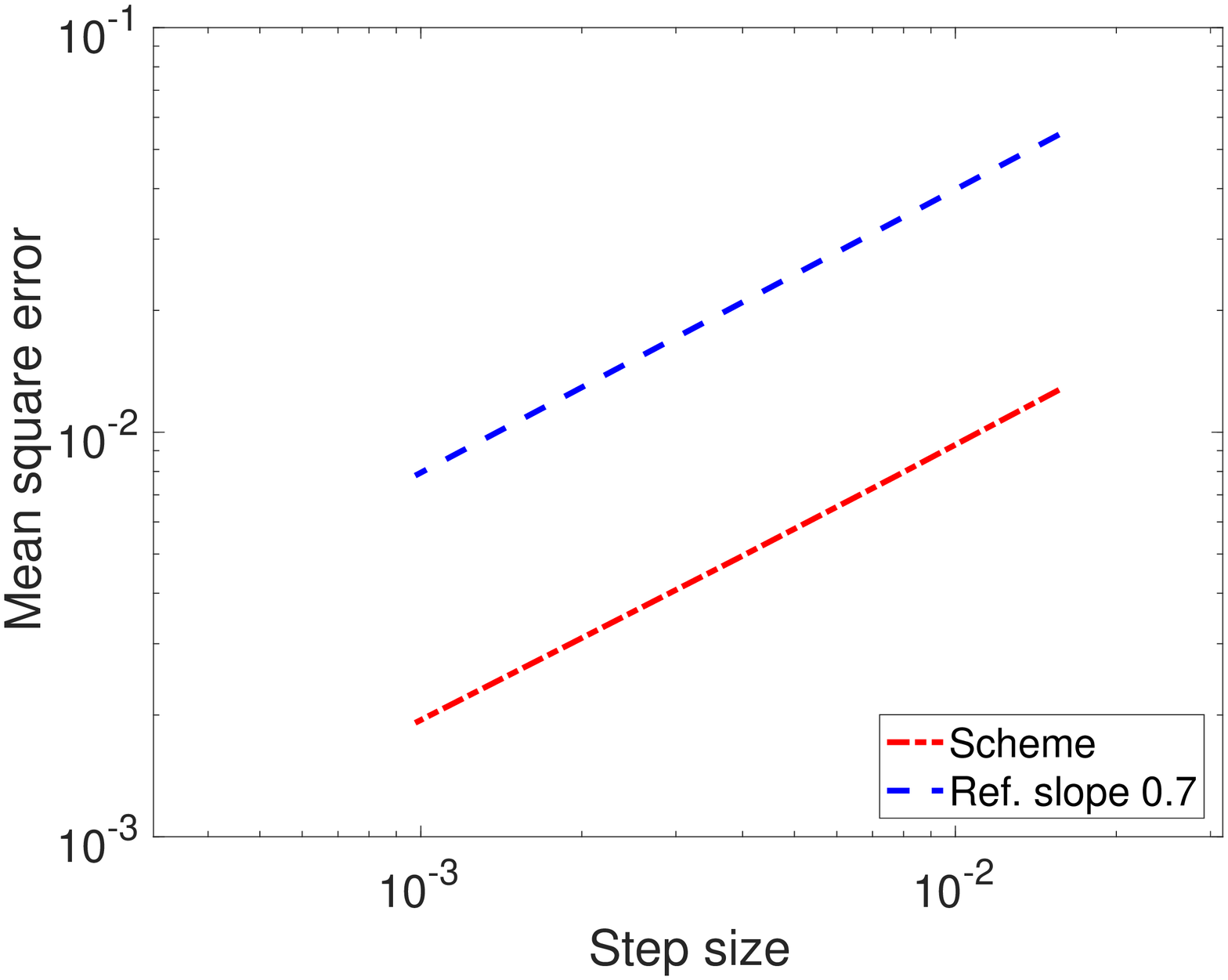}
		\end{minipage}
	}
	\subfigure[$H=0.8$]{
		\begin{minipage}[t]{0.3\linewidth}
			\includegraphics[height=3.6cm,width=3.6cm]{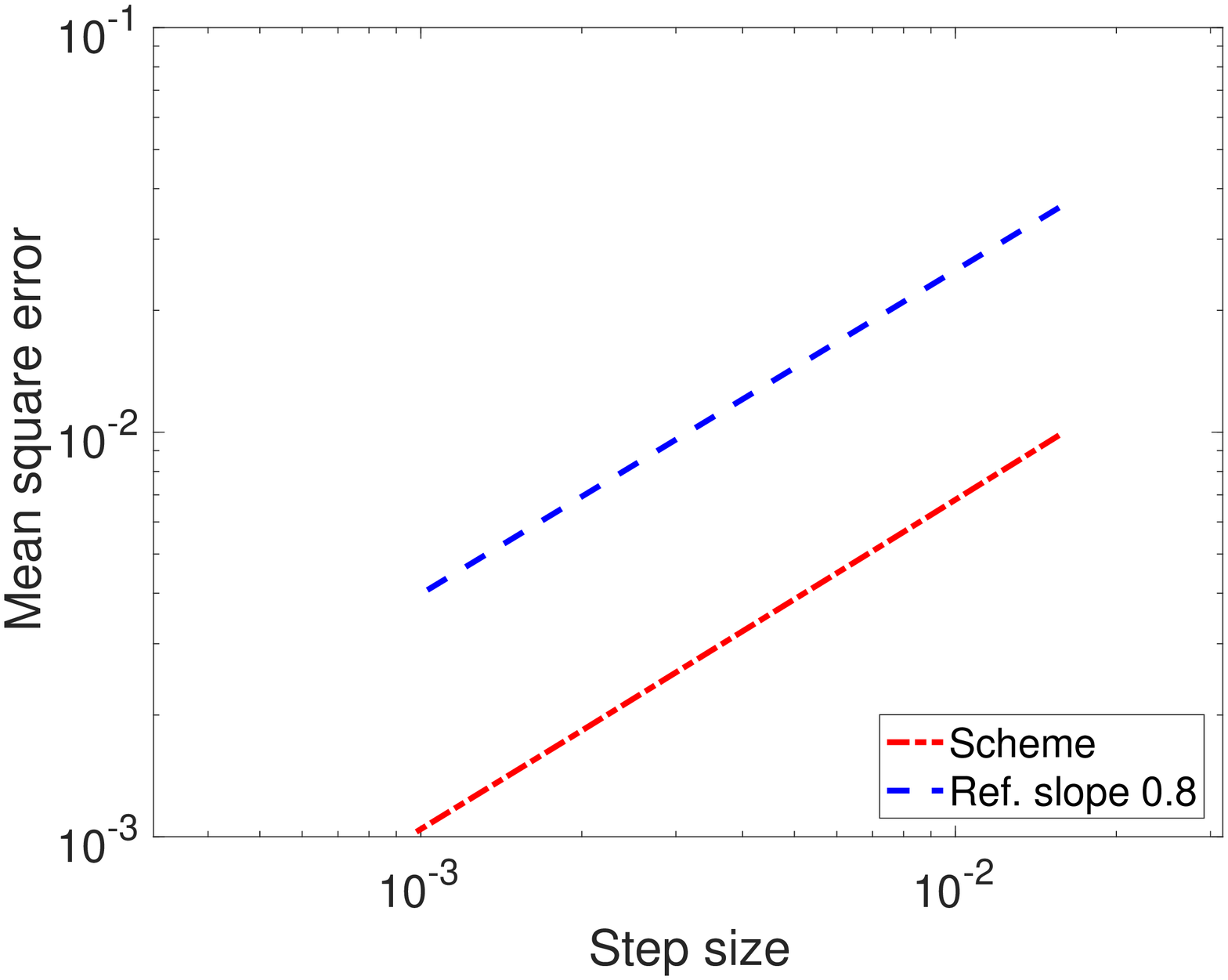}
		\end{minipage}
	}
	\caption{Mean square error for $X^h$ vs. step size }\label{f3}
\end{figure}

In Figure \ref{f4}, we take $T=10$ and $h=\frac{10}{2^{15}}$. We calculate $\|(X_n)^{-1}\|_{L^2(\Omega;\mathbb{R})}$ for $n=0,\cdots,2^{15}$, where the expectation is approximated by computing averages over 100 samples. This coincides with our theoretical analysis for the boundedness of inverse moments of numerical solutions.

\begin{figure}
	\centering
	\subfigure[$H=0.6$]{
		\begin{minipage}[t]{0.3\linewidth}
			\includegraphics[height=3.6cm,width=3.6cm]{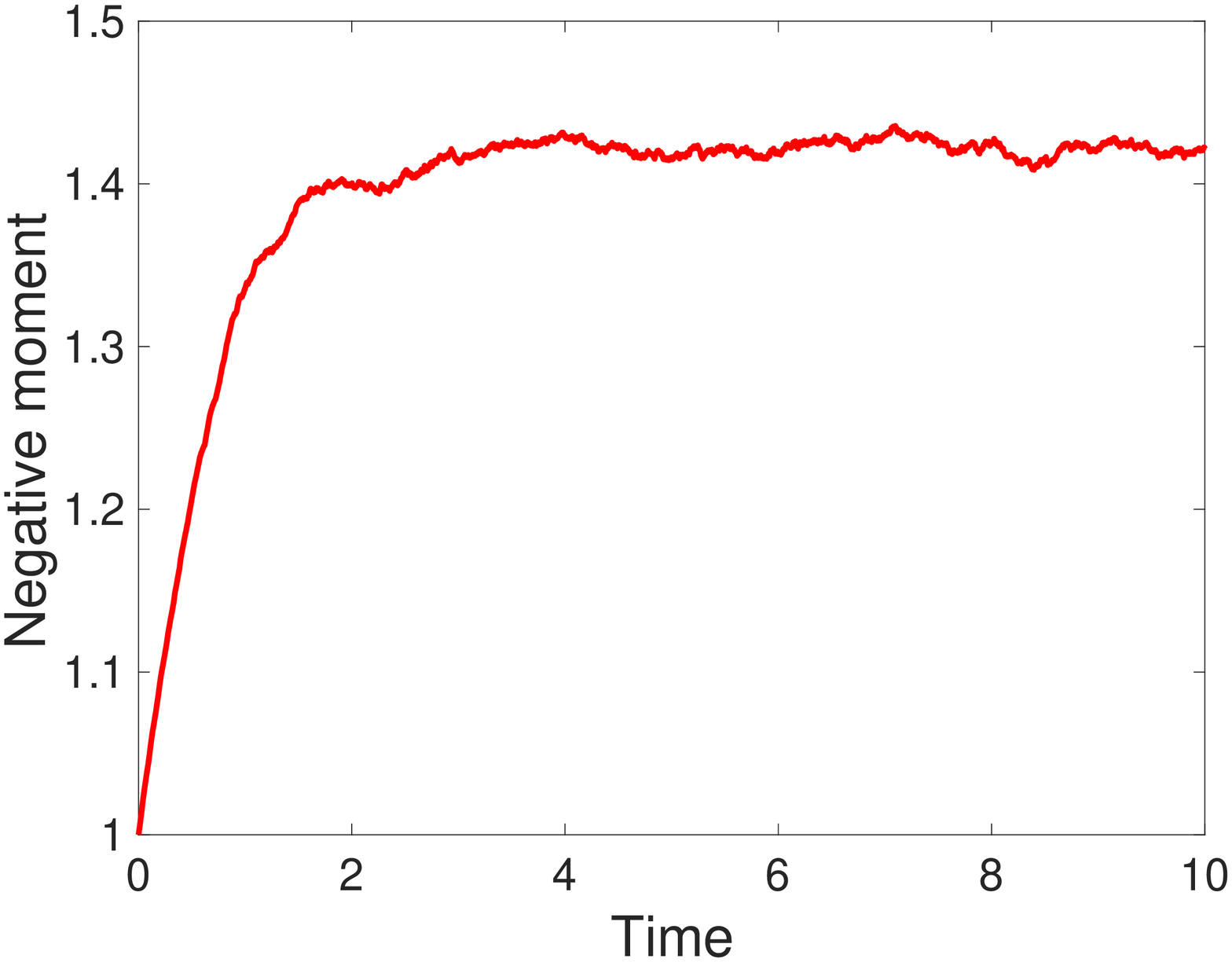}
		\end{minipage}
	}
	\subfigure[$H=0.7$]{
		\begin{minipage}[t]{0.3\linewidth}
			\includegraphics[height=3.6cm,width=3.6cm]{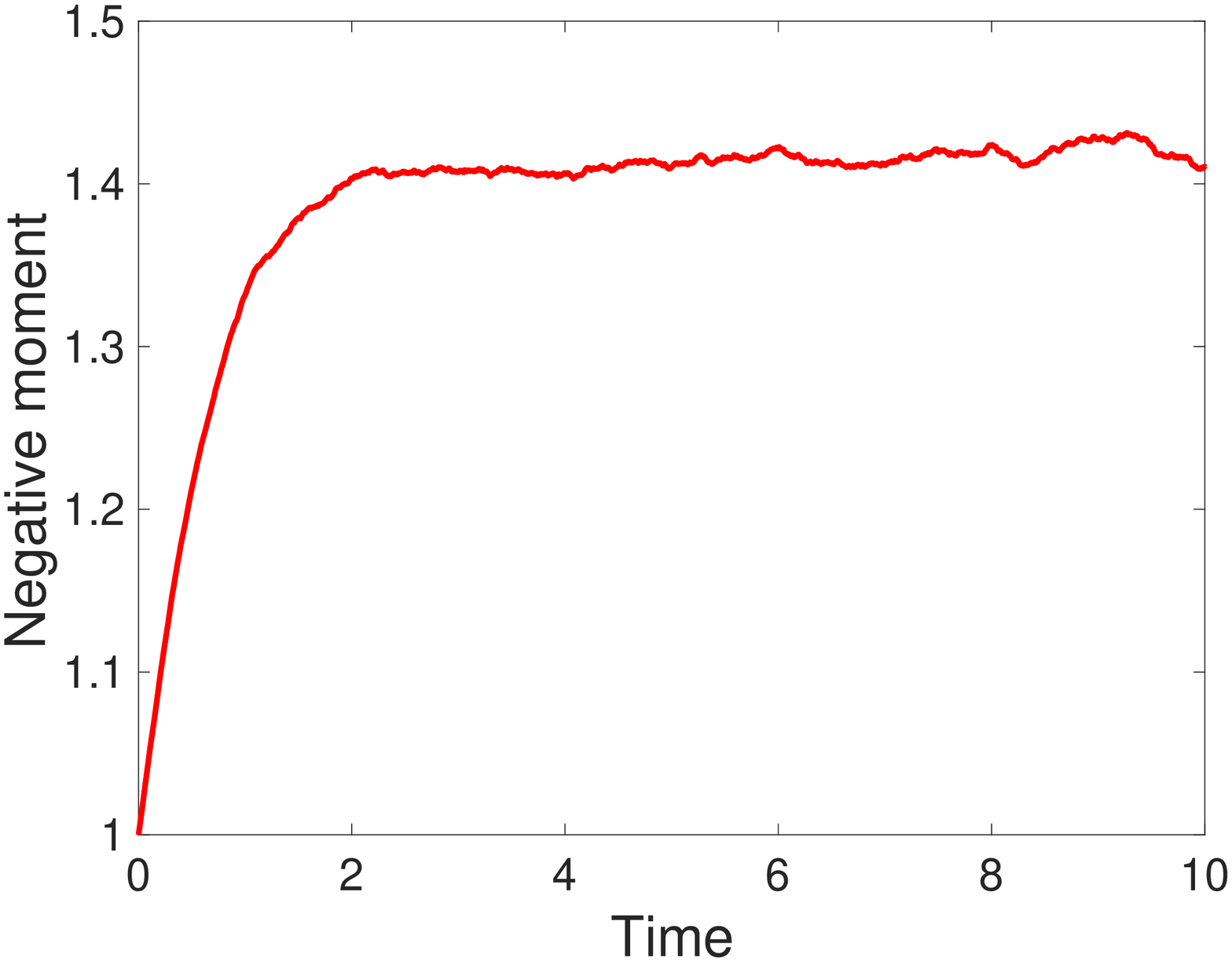}
		\end{minipage}
	}
	\subfigure[$H=0.8$]{
		\begin{minipage}[t]{0.3\linewidth}
			\includegraphics[height=3.6cm,width=3.6cm]{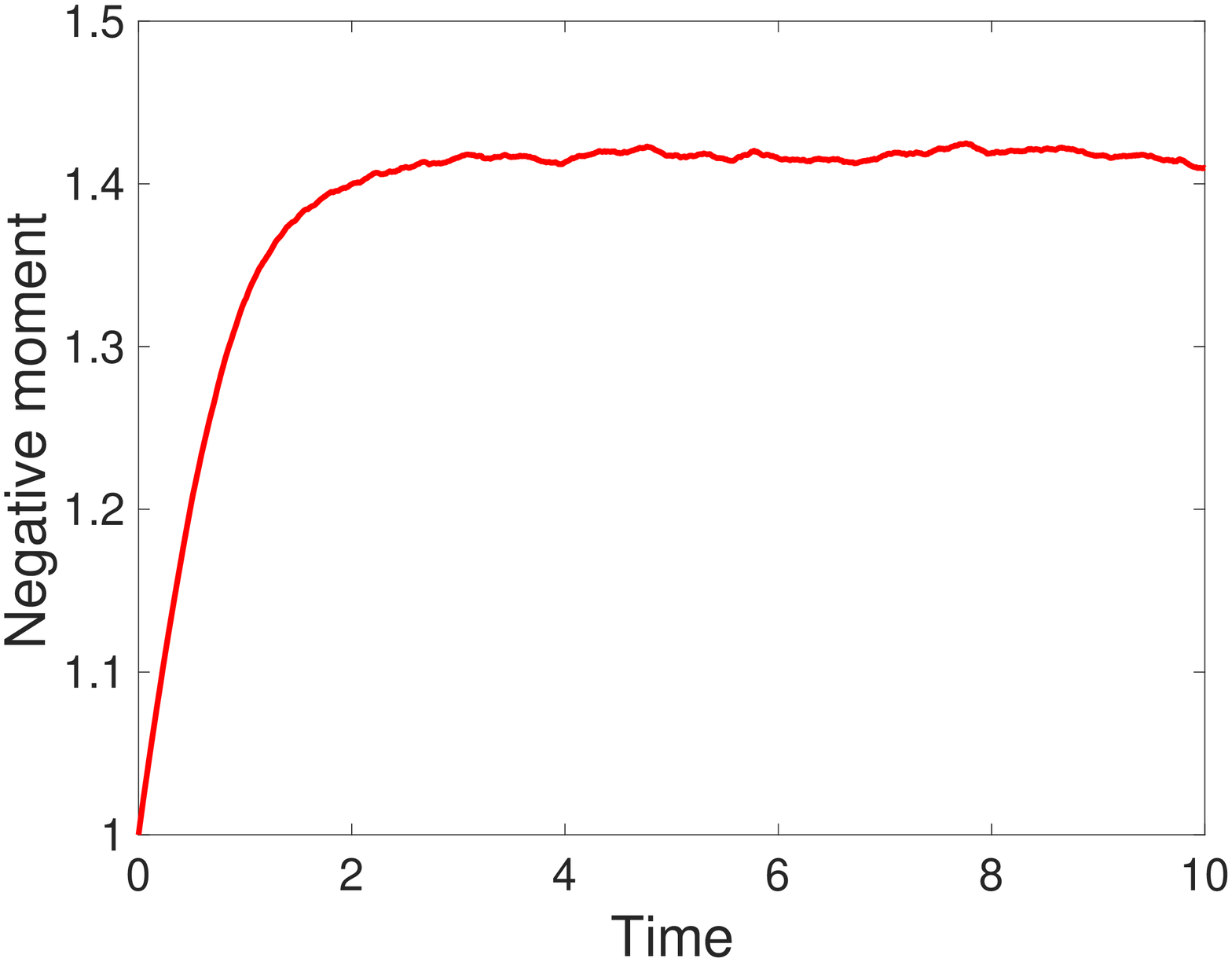}
		\end{minipage}
	}
	\caption{Inverse moment of $X_n$ vs. time}\label{f4}
\end{figure}

\section*{Acknowledgement}
The authors gratefully thank the anonymous referees for many valuable comments and suggestions in improving this paper.
This work is funded by National Natural Science Foundation of China (NO. 91530118, NO. 91130003, NO. 11021101, NO. 91630312 and NO. 11290142).


\section*{References}

\begin{thebibliography}{10}
	\bibitem{A05}
	A.~Alfonsi, On the discretization schemes for the {CIR} (and {B}essel
		squared) processes, Monte Carlo Methods Appl. 11 (2005), no.~4,
	355--384. 
	
	\bibitem{A13}
	A.~Alfonsi, Strong order one convergence of a drift implicit {E}uler scheme:
		application to the {CIR} process, Statist. Probab. Lett. 83 (2013),
	no.~2, 602--607. 
	
	\bibitem{BBD08}
	A.~Berkaoui, M.~Bossy, A.~Diop,  Euler scheme for {SDE}s with
		non-{L}ipschitz diffusion coefficient: strong convergence, ESAIM Probab.
	Stat.  12 (2008), 1--11. 
	
	\bibitem{CIR85}
	J.~C. Cox, J.~E. Ingersoll, S.~A. Ross,  A theory of the term
		structure of interest rates, Econometrica  53 (1985), no.~2,
	385--407. 
	
	\bibitem{DU99PA}
	L.~Decreusefond, A.~S. \"Ust\"unel,  Stochastic analysis of the
		fractional {B}rownian motion, Potential Anal.  10 (1999), no.~2,
	177--214. 
	
	\bibitem{DNS12}
	S.~Dereich, A.~Neuenkirch, L.~Szpruch,  An {E}uler-type method for the
		strong approximation of the {C}ox-{I}ngersoll-{R}oss process, Proc. R. Soc.
	Lond. Ser. A Math. Phys. Eng. Sci.  468 (2012), no.~2140, 1105--1115.
	
	
	\bibitem{AIHP}
	A.~Deya, A.~Neuenkirch, S.~Tindel,  A {M}ilstein-type scheme without
		{L}\'{e}vy area terms for {SDE}s driven by fractional {B}rownian motion,
	Ann. Inst. Henri Poincar\'{e} Probab. Stat.  48 (2012), no.~2,
	518--550.
	
	\bibitem{Hes18}
	O.~E. Euch, M.~Rosenbaum,  The characteristic function of rough heston
		models, Mathematical Finance (2018).
	
	\bibitem{AMO09}
	M.~J. Garrido-Atienza, P.~E. Kloeden, A.~Neuenkirch,  Discretization
		of stationary solutions of stochastic systems driven by fractional {B}rownian
		motion, Appl. Math. Optim.  60 (2009), no.~2, 151--172. 
	
	\bibitem{GR11}
	I.~Gy\"ongy, M.~R\'asonyi,  A note on {E}uler approximations for {SDE}s
		with {H}\"older continuous diffusion coefficients, Stochastic Process. Appl.
	 121 (2011), no.~10, 2189--2200. 
	
	\bibitem{Heston93}
	S.~L. Heston,  A closed-form solution for options with stochastic
		volatility with applications to bond and currency options, J. Math. Kyoto
	Univ.  6 (1993), no.~2, 327--343.
	
	\bibitem{HMS13}
	D.~J. Higham, X.~Mao, L.~Szpruch,  Convergence, non-negativity and
		stability of a new {M}ilstein scheme with applications to finance, Discrete
	Contin. Dyn. Syst. Ser. B  18 (2013), no.~8, 2083--2100. 
	
	\bibitem{HHW}
	J.~Hong, C.~Huang, X.~Wang,  Strong convergence rate of
		{R}unge--{K}utta methods and simplified step-{$N$} {E}uler schemes for {SDE}s
		driven by fractional {B}rownian motions arXiv: 1711.02907v2.
		
		
	
	\bibitem{CN2017}
	Y.~Hu, Y.~Liu, D.~Nualart,  Crank-{N}icolson scheme for stochastic
		differential equations driven by fractional {B}rownian motions, arXiv:
	1709.01614.
	
	\bibitem{HLN16AAP}
	Y.~Hu, Y.~Liu, D.~Nualart,  Rate of convergence and asymptotic error distribution of {E}uler
		approximation schemes for fractional diffusions, Ann. Appl. Probab.
	 26 (2016), no.~2, 1147--1207. 
	
	\bibitem{HNS08}
	Y.~Hu, D.~Nualart, X.~Song,  A singular stochastic differential
		equation driven by fractional {B}rownian motion, Statist. Probab. Lett.
	 78 (2008), no.~14, 2075--2085. 
	 
	 \bibitem{AAP03PLE}
	 J.~H\"usler, V.~Piterbarg, O.~Seleznjev, On convergence of the
	 	uniform norms for {G}aussian processes and linear approximation problems,
	 Ann. Appl. Probab. 13 (2003), 1615--1653,
	
	\bibitem{JKN09NM}
	A.~Jentzen, P.~E. Kloeden, A.~Neuenkirch,  Pathwise approximation of
		stochastic differential equations on domains: higher order convergence rates
		without global {L}ipschitz coefficients, Numer. Math.  112 (2009),
	no.~1, 41--64.
	\bibitem{MC11}
	P.~E. Kloeden, A.~Neuenkirch, R.~Pavani,  Multilevel {M}onte {C}arlo
		for stochastic differential equations with additive fractional noise, Ann.
	Oper. Res.  189 (2011), 255--276. 
	
	\bibitem{Rev68}
	B.~B. Mandelbrot, J.~W. Van~Ness,  Fractional {B}rownian motions,
		fractional noises and applications, SIAM Rev.  10 (1968), 422--437.
	
	\bibitem{MT04}
	G.~N. Milstein, M.~V. Tretyakov,  Stochastic numerics for mathematical
		physics, Scientific Computation, Springer-Verlag, Berlin, 2004. 
	
	\bibitem{Mishura1804}
	Y.~Mishura, A.~Yurchenko-Tytarenko,  Fractional
		{C}ox-{I}ngersoll-{R}oss process with non-zero mean, Mod. Stoch. Theory Appl. 5 (2018), no.~1, 99--111. 
		
	\bibitem{N06}
	A.~Neuenkirch,  Optimal approximation of {SDE}'s with additive fractional
		noise, J. Complexity  22 (2006), no.~4, 459--474. 
	
	\bibitem{NS14}
	A.~Neuenkirch, L.~Szpruch,  First order strong approximations of scalar
		{SDE}s defined in a domain, Numer. Math.  128 (2014), no.~1,
	103--136. 
	
	\bibitem{Nualart}
	D.~Nualart,  The {M}alliavin calculus and related topics, second ed.,
	Probability and its Applications (New York), Springer-Verlag, Berlin, 2006.

	
	\bibitem{S85}
	H.~Sugita,  On a characterization of the {S}obolev spaces over an abstract
		{W}iener space, J. Math. Kyoto Univ.  25 (1985), no.~4, 717--725.

	

	\end{thebibliography}
\end{document}